\documentclass[12pt]{amsart}
\usepackage{epsfig,color}
\usepackage{blindtext}
\usepackage{graphicx}
\usepackage{float}
\usepackage[shortlabels]{enumitem}
\usepackage{url}
\usepackage{amssymb}
\usepackage{graphicx,import}
\usepackage{comment}
\usepackage{esint}
\usepackage{xcolor}
\usepackage{mathtools}
\usepackage{comment}
\usepackage{tikz}
\usepackage{thmtools}
\usepackage{thm-restate}
\usepackage{hyperref}
\usepackage{cleveref}

\usepackage[margin = 1in] {geometry}

\usepackage{hyperref}
\usepackage{dsfont}

\newtheorem{theorem}{Theorem}[section]
\newtheorem{conjecture}[theorem]{Conjecture}
\newtheorem{proposition}[theorem]{Proposition}
\newtheorem{lemma}[theorem]{Lemma}
\newtheorem{corollary}[theorem]{Corollary}

\theoremstyle{definition}

\newtheorem{definition}[theorem]{Definition}
\newtheorem*{remark*}{Remark}

\newtheorem{remark}[theorem]{Remark}

\numberwithin{equation}{section}

\newcommand{\cH}{\mathcal{H}}
\newcommand{\R}{\mathbb{R}}
\newcommand{\RP}{\mathbb{R}\mathbb{P}}

\newcommand{\N}{\mathbb{N}}
\newcommand{\Z}{\mathbb{Z}}

\newcommand{\eps}{\varepsilon}

\newcommand{\Ric}{\mathrm{Ric}}

\newcommand{\length}{\mathrm{length}}

\newcommand{\red}[1]{\textcolor{red}{#1}}
\newcommand{\blue}[1]{\textcolor{blue}{#1}}
\newcommand{\purple}[1]{\textcolor{purple}{#1}}
\newcommand{\nl}{\newline}

\newcommand{\p}{\partial}

\newcommand{\Geo}{\textbf{\text{Geo}}}
\newcommand{\ord}{\text{ord}}
\newcommand{\sauce}{\text{\textcolor{red}{source}}}

\title{Index, Intersections, and Multiplicity of Min-Max Geodesics}
\author{Jared Marx-Kuo, Lorenzo Sarnataro, Douglas Stryker}
\begin{document}

\begin{abstract}
    We prove upper bounds for the Morse index and number of intersections of min-max geodesics achieving the $p$-width of a closed surface. A key tool in our analysis is a proof that for a generic set of metrics, the tangent cone at any vertex of any finite union of closed immersed geodesics consists of exactly two lines. 
 
    We also construct examples to demonstrate that multiplicity one does not hold generically in this setting. Specifically, we construct an open set of metrics on $S^2$ for which the $p$-width is only achieved by $p$ copies of a single geodesic.
\end{abstract}
\maketitle
\tableofcontents
\section{Introduction}
For a closed Riemannian manifold $(M^{n+1}, g)$ of dimension $n+1$, the \emph{volume spectrum} is a sequence of geometric invariants $\{\omega_p(M, g)\}_{p\in\N}$ introduced by Gromov \cite{gromov2002isoperimetry, gromov2006dimension, gromov2010singularities}, called the \emph{$p$-widths}. This sequence is a nonlinear analog of the spectrum of the Laplacian operator. The volume spectrum plays an essential role in many significant breakthroughs in the study of minimal hypersurfaces, including the resolution of Yau's conjecture on the existence of infinitely many closed embedded minimal hypersurfaces in any closed ambient manifold $(M^{n+1}, g)$ of dimension $3 \leq n+1 \leq 7$ achieved by \cite{song2018existence}. We refer the reader to \cite{almgren1962homotopy, pitts1981existence, marques2014min, MN_index_upper, marques2017existence,  liokumovich2018weyl, irie2018density, marques2021morse, zhou2020multiplicity, gaspar2018allen, gaspar2019weyl, gaspar2020index, chodosh2020minimal, Dey} for the historical developments of this program.

The reason for the utility of the volume spectrum in the study of minimal hypersurfaces is the fact that (when $3 \leq n+1 \leq 7$) each $p$-width equals the weighted area of a smooth closed embedded \emph{min-max} minimal hypersurface: namely, there are disjoint connected smooth closed embedded minimal hypersurfaces $\{\Sigma_{p,j}\}_{j=1}^{N(p)}$ and positive integers $\{m_{p,j}\}_{j=1}^{N(p)}$ so that
\begin{equation}\label{eqn:rep_intro} \omega_p(M, g) = \sum_{j=1}^{N(p)} m_{p,j}\mathrm{Area}(\Sigma_{p,j}). \end{equation}

By analogy with classical Morse theory, the results of \cite{MN_index_upper} imply that the minimal hypersurface achieving the $p$-width from \eqref{eqn:rep_intro} can be chosen to additionally satisfy the index bound
\begin{equation}\label{eqn:index_intro} \sum_{j=1}^{N(p)} \mathrm{index}(\Sigma_{p,j}) \leq p, \end{equation}
where $\mathrm{index}(\Sigma)$ is the Morse index of the minimal hypersurface $\Sigma$ (meaning the maximal dimension of a linear subspace of normal variations on which the second variation of area is negative definite).

Due to an example of \cite{WangZhou_mult2}, the possibility of multiplicities $m_{p,j}$ greater than 1 in \eqref{eqn:rep_intro} is unavoidable in general. However, as a consequence of the resolution of the multiplicity one conjecture due to \cite{chodosh2020minimal, zhou2020multiplicity}, which asserts that $\Sigma_{p,j}$ is 2-sided and the multiplicities $m_{p,j}$ are equal to 1 for a generic set of metrics on $M$, the minimal hypersurface achieving the $p$-width can be chosen to satisfy the stronger \emph{weighted} index bound
\begin{equation}\label{eqn:weighted_index_intro} \sum_{\Sigma_{p,j}\ \text{2-sided}} m_{p,j}\mathrm{index}(\Sigma_{p,j}) + \sum_{\Sigma_{p,j}\ \text{1-sided}} \frac{m_{p,j}}{2}\mathrm{index}(\Sigma_{p,j}) \leq p. \end{equation}

These results on the Morse index and multiplicity of representatives of the $p$-widths only apply to ambient manifolds of dimension $3 \leq n+1 \leq 7$. The dimension upper bound is due to the existence of singularities. The dimension lower bound is due to the fact that min-max on surfaces may only produce stationary geodesic networks in general (see \cite{pitts1973net}), in which case the techniques to control index and multiplicity fail (see \cite[Remark 1.1]{MN_index_upper}). However, in a recent breakthrough, Chodosh--Mantoulidis used the Allen--Cahn min-max framework (with the \emph{sine-Gordon} potential) to show that when $n+1 = 2$, the $p$-widths are realized by unions of closed immersed geodesics:

\begin{restatable}[\cite{chodosh2023p}]{thmm}{CMGeodesicThm}\label{CMGeodesicThm}
Let $(M^2, g)$ be a closed Riemannian surface. For every $p \in \N$, there is a collection of closed immersed geodesics $\{\sigma_{p, j}\}_{j=1}^{N(p)}$ and positive integers $\{m_{p,j}\}_{j=1}^{N(p)}$ satisfying 
\begin{equation}\label{eqn:rep_curves}
\omega_p(M, g) = \sum_{j = 1}^{N(p)} m_{p,j}\mathrm{length}_g(\sigma_{p,j}).
\end{equation}
\end{restatable}

Using this regularity theory as a starting point, we investigate extensions of the various well-known aspects of Morse index and multiplicity in min-max theory to the setting of curves in surfaces.

An important preliminary observation is that these min-max geodesics are not necessarily embedded, and self-intersections are expected to contribute to a proper accounting of index or instability in this setting. We let $\mathrm{Vert}(\{\sigma_{p,j}\}_j)$ denote the set of points $v \in M^2$ so that the tangent cone of $\bigcup_j \sigma_{p,j}(S^1)$ at $v$ consists of at least two distinct lines, and we let $\mathrm{ord}(v)$ denote the number of distinct lines in the tangent cone of $\bigcup_j \sigma_{p,j}(S^1)$ at $v \in \mathrm{Vert}(\{\sigma_{p,j}\}_j)$.

Specifically, we address the following two questions as generalizations of \eqref{eqn:index_intro} and \eqref{eqn:weighted_index_intro}:
\begin{enumerate}
    \item \label{ite:Q1} \cite[Open Question 1]{chodosh2023p} Can the union of closed immersed geodesics achieving the $p$-width in \eqref{eqn:rep_curves} be chosen to satisfy
    \begin{equation} \label{CMIndexEq}
    \sum_{j=1}^{N(p)} \mathrm{index}(\sigma_{p,j}) + \sum_{v \in \mathrm{Vert}(\{\sigma_{p,j}\}_j)} \binom{\mathrm{ord}(v)}{2} \leq p?
    \end{equation}
    \item \label{ite:Q2} Does $m_{p,j} = 1$ hold for a generic set of metrics on $M^2$? In particular, can \eqref{CMIndexEq} be upgraded to a weighted bound accounting for multiplicities, as in \eqref{eqn:weighted_index_intro}?
\end{enumerate}

\subsection{Main results}
Our main results touch on the questions just posed.

An essential ingredient in our index analysis is a generic metric theorem of independent interest. We show that generically, the order of any vertex of a finite union of closed immersed geodesics is 2 (i.e. at any intersection point, the tangent cone consists of two distinct lines in $\R^2$, 
intersecting transversely). 
In \S\ref{GenericInterSec}, we show (see Theorem \ref{GenSelfIntFull} for the full statement and Figure \ref{fig:genInt} for visualization):
%
\begin{restatable}{thmm}{GenSelfIntThm}\label{GenSelfIntThm}
    Let $M^2$ be a closed surface. The set of metrics on $M^2$ with the following property is $C^k$-generic in the Baire sense for all $k \geq 3$: the support of the tangent cone at any point of any finite union of closed immersed geodesics consists of at most two distinct lines.
%
\end{restatable}
%


%
We can use \Cref{GenSelfIntThm} to provide a partial answer to Question \ref{ite:Q1}:
%
\begin{restatable}{thmm}{thm:index_upper}\label{thm:index_upper}
    Let $(M^2, g)$ be a closed smooth Riemannian surface. For every $p \in \N$, there is a collection of closed immersed geodesics $\{\sigma_{p, j}\}_{j=1}^{N(p)}$ and positive integers $\{m_{p,j}\}_{j=1}^{N(p)}$ satisfying
    \begin{equation*}\label{eq:achieve_width}
    \omega_p(M, g) = \sum_{j=1}^{N(p)} m_{p,j}\mathrm{length}_g(\sigma_{p,j})
    \end{equation*}
    and
    \begin{align*} 
    (a)& \quad\quad\ \, \sum_{j = 1}^{N(p)} \mathrm{index}(\sigma_{p,j}) \leq p, \\ 
    (b)& \quad \sum_{v \in \mathrm{Vert}(\{\sigma_{p,j}\}_j)} \begin{pmatrix}\mathrm{ord}(v) \\ 2 \end{pmatrix} \leq p.
    \end{align*}
    %
\end{restatable}
%
We remark that equation (b) shows that the union $\{\sigma_{p,j}\}_j$ has at most $p$ vertices. However, we emphasize that the stronger, full bound of equation (b) is a consequence of working in the generic setting afforded by \Cref{GenSelfIntThm}. We hope to address the entirety of equation \eqref{CMIndexEq} in future work.
%
%

We emphasize that the bounds in \Cref{thm:index_upper} do not account the multiplicities $m_{p,j}$. Indeed, we demonstrate that higher multiplicity can occur in an open set of metrics, providing a negative answer to Question \ref{ite:Q2}:
\begin{restatable}{thmm}{HigherMultCounter} \label{HigherMultCounter}
For any $p \in \N$, there exists an open set of metrics $U_p$ on $S^2$, such that for any $g \in U_p$
\begin{enumerate}
    \item $K_g > 0$,
    \item $\omega_l(S^2, g) = 2 \pi l$ for all $l =1,\dots, p,$
    \item each $\omega_l$ can only be achieved by $l$ copies of the same, nondegenerate geodesic for $l = 1,\hdots, p$.
\end{enumerate}
\end{restatable}
Our construction essentially follows the ideas of \cite{WangZhou_mult2}. However, the surface setting does not require the reverse catenoid estimate for degenerate stable minimal hypersurfaces employed by \cite{WangZhou_mult2}, which allows us to construct an example with positive curvature. Since the proof of the weighted index bound \eqref{eqn:weighted_index_intro} relies on the generic multiplicity one theorem of \cite{zhou2020multiplicity}, a weighted version of \eqref{CMIndexEq} cannot be proved by the same method.

%

In \cite{aiex2016width}, an example is constructed of an ellipsoid metric near the round sphere with the property that at least one of $\omega_4(S^2, g)$, $\omega_5(S^2, g)$, and $\omega_6(S^2, g)$ can only be achieved by a multiplicity 2 geodesic. We mention a few novel feature of our construction:
\begin{itemize}
    \item Our conclusion holds for an open set of metrics, which provides a counterexample to a generic multiplicity one result.
    \item Our construction is the first example guaranteeing multiplicity larger than 2.
    \item We find higher multiplicity for $\omega_2$ and $\omega_3$, not only for the higher widths.
    \item Our example exhibits the \emph{worst-case behavior}: the sweepout construction of \cite{guth2009minimax} and the results of \cite{CalabiCao92} imply that the multiplicity of a geodesic achieving the $p$-width for any positive curvature metric on $S^2$ is at most $p$. Our construction therefore provides an example where the multiplicity is maximal.
\end{itemize}

Moreover, our constructions pass to $\R \mathbb{P}^2$ (analogous to \cite[Cor 1.2]{WangZhou_mult2}):
\begin{restatable}{corr}{rptwoCorr} \label{rptwoCorr}
For any $p \in \N$, there exists an open set of metrics $U_p$ on $\R \mathbb{P}^2$, such that for any $g \in U_p$
\begin{enumerate}
    \item $K_g > 0$,
    \item $\omega_l(S^2, g) = 2 \pi l$ for all $l =1,\dots, p,$
    \item each $\omega_l$ can only be achieved by $2l$ copies of the same, nondegenerate, one-sided geodesic for $l = 1,\hdots, p$.
\end{enumerate}
\end{restatable}

\subsection{Main ideas}
We sketch the main ideas of our results.

The main new ingredient in \Cref{GenSelfIntThm} is an iterated conformal deformation process. Suppose $\gamma$ is a geodesic in the metric $g$, and let $\tilde{\gamma}$ be a small normal graph over $\gamma$ such that $\gamma$ agrees with $\tilde{\gamma}$ outside a small geodesic ball $B_r(x)$. Then we find a small conformal deformation $\tilde{g}$ of $g$ supported in $B_r(x)$ so that $\tilde{\gamma}$ is a geodesic in the metric $\tilde{g}$. Iterations of this construction allow us to manually decrease the orders of the intersections of geodesics.


The index upper bound in dimensions $3 \leq n+1 \leq 7$ due to \cite[Theorem A]{gaspar2020index} requires embeddedness in an important way (in the same way that embeddedness is required for the index upper bound in the Almgren--Pitts setting due to \cite{MN_index_upper}): normal vector fields along an embedded hypersurfaces can be extended to vector fields on the ambient manifold. To prove the index upper bound of \Cref{thm:index_upper}, we make two observations. 
\begin{enumerate}
    \item An approximation argument using \Cref{GenSelfIntThm} allows us to work with unions of closed immersed geodesics that only have order 2 vertices.
    \item By a simple observation in linear algebra, up to the addition of a tangential vector field, a normal vector field along a union of closed immersed geodesics with order 2 vertices can be extended to an ambient vector field.
\end{enumerate}
Since the second variation of length only depends on the normal component of a variation, the argument of \cite[Theorem A]{gaspar2020index} can be extended to our setting. We emphasize that we have to take special care in constructing the extended ambient vector field to ensure that the error term in \cite[Proposition 3.3]{gaspar2020index} is small. In the embedded case, this error term vanishes automatically. While we cannot guarantee that it vanishes in our setting, we can find extensions that make the error term arbitrarily small.

By a result of \cite{tonegawa2005stable}, the number of vertices of any geodesic network produced by min-max for the $p$-width is at most $p$. By an approximation argument using \Cref{GenSelfIntThm}, we can upgrade this vertex bound obtained in the generic setting to the general bound in \Cref{thm:index_upper}. 

For \Cref{HigherMultCounter} and \Cref{rptwoCorr}, we adapt an example of \cite{WangZhou_mult2} to the surface setting. While Wang--Zhou require a delicate ``reverse-catenoid" estimate to show that $\omega_2 = 2 \omega_1$, this is not necessary on surfaces, and higher multiplicity follows from the Frankel property of our surface and compactness of geodesics of bounded length. 

\subsection{Paper organization}
This paper is organized as follows:
\begin{itemize}
    \item In \S\ref{sec:background}, we review the required background results about min-max. 
    
    \item In \S \ref{prelims}, we establish notation and terminology, as well as some simple lemmas and constructions with geodesics.

    \item In \S \ref{GenericInterSec}, we prove \Cref{GenSelfIntThm}

    \item In \S \ref{IndexSec}, we prove \Cref{thm:index_upper} 

    \item In \S \ref{HighMultSec}, we prove \Cref{HigherMultCounter} and \Cref{rptwoCorr} and provide further examples.
\end{itemize}

\subsection{Acknowledgements}
The authors are thankful to Otis Chodosh for suggesting this project, as well as Christos Mantoulidis, Fernando Marques, and Akashdeep Dey for insightful conversations. 
%
%
%
%
\section{Background}\label{sec:background}
In this section we review the terminology and notation of min-max theory, and give a rigorous definition of the $p$-widths (both in the Almgren--Pitts and Allen--Cahn settings), which play a central role in this paper, as mentioned in the Introduction. We follow the presentation in \cite[\S2]{chodosh2023p}, to which we refer the reader for more details. In this section, we will make use of some standard notation from geometric measure theory, which we briefly recap below. We refer to \cite{marques2014min} and \cite{Simon1983} for the relevant definitions.

For the rest of this section, $(M^2, g)$ will denote a fixed closed Riemannian surface.

\subsection{Notation}
\begin{itemize}
    \item $\mathbf{I}_k(M;\Z_2)$: the space of $k$-dimensional mod 2 flat chains in $M$, equipped with the topology induced by the \textit{flat metric} $\mathcal{F}$;
    \item $\mathcal{Z}_1(M;\Z_2)$: the space of 1-dimensional flat cycles, i.e. 1-dimensional flat chains $T\in\mathbf{I}_1(M;\Z_2)$ such that $T=\partial\Omega$ for some $\Omega\in\mathbf{I}_2(M;\Z_2)$; 
    \item $\mathbf{M}$: the mass functional on $\mathbf{I}_k(M;\Z_2)$; if $\gamma$ denotes the 1-cycle induced by the submanifold $\gamma(S^1)$ for an immersed closed curve $\gamma:S^1\to M$, then $\mathbf{M}(\gamma)=\mathrm{length}(\gamma)$;
    \item $G_1(M)=\{(x,P): x\in M, P\in G(T_xM, 1)\}$, where $G(T_xM, 1)$ denotes the space of unoriented 1-dimensional subspaces in $T_xM$;
    \item $\mathcal{V}_1(M)$: the space of 1-varifolds on $M$ (i.e. of Radon measures on $G_1(M)$), equipped with the topology induced by the \textit{varifold metric} $\mathbb{F}$;
    \item $\mathcal{IV}_1(M)$: the space of integral 1-varifolds on $M$,
    \item $\|V\|$: the Radon measure induced on $M$ by $V\in\mathcal{V}_1(M)$,
    \item $|T|, |\gamma|$: the integral 1-varifold induced by a mod 2 flat chain $T\in\mathbf{I}_1(M;\Z_2)$, or by the submanifold $\gamma(S^1)$ for an immersed closed curve $\gamma:S^1\to M$;
\end{itemize}
As customary in min-max theory, we shall also use $\mathbf{F}$ to denote the so-called $\mathbf{F}$-metric on $\mathbf{I}_1(M;\Z_2)$ defined by
\begin{equation*}
    \mathbf{F}(S,T)\coloneqq\mathcal{F}(S,T)+\mathbf{F}(|S|,|T|)
\end{equation*}
for $S,T\in\mathbf{I}_1(M;\Z_2)$.

It is worth noticing at this point that by \cite{marques2017existence}, $\mathbf{I}_2(M;\Z_2)$ is contractible, and the boundary map 
\[\partial: \mathbf{I}_2(M;\Z_2)\to\mathcal{Z}_1(M;\Z_2)\]
is a double cover.

\subsection{Almgren--Pitts min-max theory and the volume spectrum} Since minimal hypersurfaces (or, in our case, closed geodesics) are critical points of the area functional (in our case, length), it is natural to mimic the ideas of Morse theory in order to construct these critical points.

The starting point of Almgren--Pitts min-max theory is Almgren's isomorphism theorem \cite{almgren1962homotopy}, which shows that $\mathcal{Z}_1(M;\Z_2)$ equipped with the flat topology is weakly homotopy equivalent to $\RP^\infty$, so that 
\[H^p(\mathcal{Z}_1(M;\Z_2);\Z_2)=\{0,\bar{\lambda}^p\}\cong\Z_2\]
for all $p\geq 1$, where $\bar{\lambda}$ is the generator of $H^1(\mathcal{Z}_1(M;\Z_2);\Z_2)$.
\begin{definition}
An $\mathbf{F}$-continuous map $\Phi:X\to\mathcal{Z}_1(M;\Z_2)$ from a finite dimensional cubical complex $X$ into the space $\mathcal{Z}_1(M;\Z_2)$ of 1-cycles in $M$ is said to be a $p$\textit{-sweepout} if $\Phi^*(\bar{\lambda}^p)\neq 0$ in $H^p(X;\Z_2)$.
\end{definition}

\begin{definition}
Let $\mathcal{P}_p(M)$ denote the set of all $p$-sweepouts, then we can define the (Almgren--Pitts) $p$\textit{-width} $\omega_p(M,g)$ by
    \begin{equation*}
        \omega_p(M,g)\coloneqq \inf_{\Phi\in\mathcal{P}_p(M)}\sup_{x\in\mathrm{dmn}(\Phi)}\mathbf{M}(\Phi(x)).
    \end{equation*}
    The volume spectrum of $(M,g)$ is the sequence $\{\omega_p(M,g)\}_{p\in \N}$\footnote{By \cite[\S1.5]{MN_index_upper}, the value of $\omega_p(M, g)$ is unchanged if we only consider $p$-sweepouts whose domain is a cubical complex of dimension $p$.}.
\end{definition}

We shall use the following lemma throughout the paper.
\begin{lemma}[{\cite[Lemma 2.1]{irie2018density}}]
The $p$-width $\omega_p(M,g)$ depends continuously on the metric $g$ with respect to the $C^0$-topology.
\end{lemma}

More generally, given a $p$-sweepout $\Phi:X\to\mathcal{Z}_1(M;\Z_2)$, one can consider its homotopy class in the following sense.

\begin{definition}
Let $\Phi:X\to\mathcal{Z}_1(M;\Z_2)$ be a $p$-sweepout. We define the (Almgren--Pitts) \textit{homotopy class} $\Pi$ of $\Phi$ to be the set
\begin{equation*}
    \Pi\coloneqq\{\mathbf{F}\text{-continuous } \Phi':X\to\mathcal{Z}_1(M;\Z_2): \Phi'\text{ is homotopic to }\Phi\text{ in the }\mathcal{F}\text{-topology}\}.
\end{equation*}
The \textit{Almgren--Pitts width} $\mathbf{L}_{\mathrm{AP}}(\Pi)$ of the homotopy class $\Pi$ is defined to be
\[\mathbf{L}_{\mathrm{AP}}(\Pi)\coloneqq \inf_{\Phi\in\Pi}\sup_{x\in X}\mathbf{M}(\Phi(x)).\]
\end{definition}

\begin{definition}
Given a homotopy class $\Pi$, we say
\begin{itemize}
    \item a sequence $\{\Phi_i:X\to\mathcal{Z}_1(M;\Z_2)\}_i\subset\Pi$ is \textit{minimizing} if 
    \begin{equation*}
        \limsup_{i\to\infty}\sup_{x\in X}\mathbf{M}(\Phi_i(x))=\mathbf{L}_\mathrm{AP}(\Pi);
    \end{equation*}
    \item a varifold $V\in\mathcal{V}_1(M)$ is in the (Almgren--Pitts) \textit{critical set} $\mathbf{C}_{\mathrm{AP}}(\{\Phi_i\})$ of a minimizing sequence $\{\Phi_i\}\subset \Pi$ if
    \begin{itemize}
        \item $\|V\|(M)=\mathbf{L}_{\mathrm{AP}}(\Pi)$,
        \item and there exists a subsequence $\{i_j\}_j$ and $\{x_j\in X\}_j$ such that
        \begin{equation*}
            |\Phi_{i_j}(x_j)|\to V
        \end{equation*}
        in the sense of varifolds.
    \end{itemize}
\end{itemize}
\end{definition}

We can summarize some of the key results of Almgren--Pitts min-max theory in ambient dimension 2 in the following statement, due to the combined work of Almgren \cite{almgren1962homotopy}, Pitts \cite{pitts1981existence, pitts1973net}, Marques--Neves \cite{marques2014min}, and Aiex \cite{aiex2016width}.
\begin{theorem}
Suppose $\mathbf{L}_{\mathrm{AP}}(\Pi)>0$. Then there exists a nontrivial stationary integral 1-varifold $V\in\mathbf{C}_\mathrm{AP}(\{\Phi^*_i\}_i)$ and finitely many points $\{p_1, \dots, p_N\}\subset M$ such that
\begin{itemize}
    \item $\|V\|(M)=\mathbf{L}_{\mathrm{AP}}(\Pi)$,
    \item $V$ has integer density everywhere,
    \item away from $\{p_i\}_{i=1}^N$, the support of $\|V\|$ is contained in a finite disjoint union of embedded geodesics,
    \item at each $p_i$, any tangent cone is a stationary geodesic network in $\R^2$, smooth away from 0.
\end{itemize}
\end{theorem}
In particular, in ambient dimension $n+1=2$, the Almgren--Pitts min-max theory for the length functional produces, in general, a geodesic network which need not be supported in a union of closed immersed geodesics (see Figure \ref{fig:networkCurve}). 
\begin{figure}[ht!]
    \centering
    \includegraphics[scale=0.3]{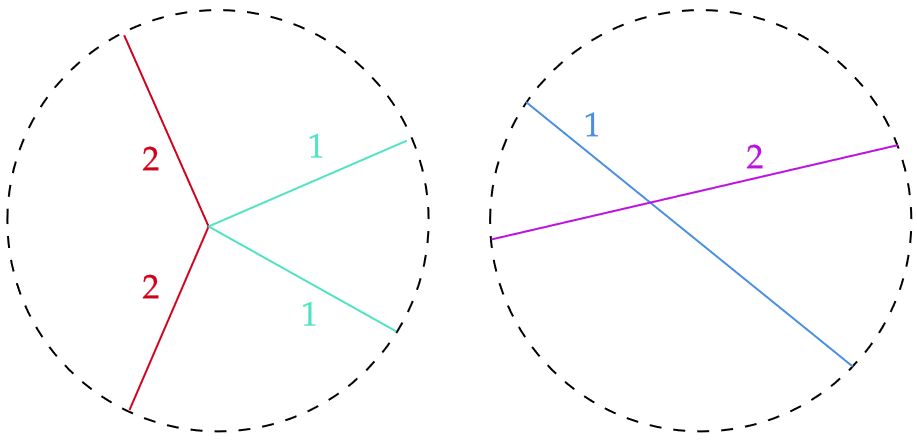}
    \caption{\textbf{L}: A geodesic network on $D^2$, which a priori, may be produced by Almgren--Pitts min-max. \textbf{R}: a union of geodesics, which can be guaranteed with Allen--Cahn min-max by \Cref{CMGeodesicThm}. }
    \label{fig:networkCurve}
\end{figure}
%
We note that not all stationary geodesic networks can arise from min-max, as Aiex \cite{aiex2016width} and Zhou--Zhu \cite{zhou2020geodesics} show that the density of the min-max 1-varifold is always an integer at junction points, using the so-called \textit{almost-minimising in annuli} property satisfied by min-max solutions. This rules out e.g. triple junctions where all segments are multiplicity one.
\subsection{Allen--Cahn min-max theory and the phase-transition spectrum}
Let $W$ be a smooth symmetric double-well potential, i.e. a smooth even function $W:\R\to [0,+\infty)$ with exactly three critical points at $-1$, 0, and 1 such that
\begin{itemize}
    \item 0 is a local maximum,
    \item $W(\pm 1)=0$,
    \item $W''(\pm 1)>0$.
\end{itemize}
The corresponding $\varepsilon$-\textit{Allen--Cahn energy} (where $\varepsilon>0$ is a small parameter) of a function $u\in L^\infty(M)\cap H^1(M)$ is given by
\begin{equation*}
    E_\varepsilon[u]\coloneqq\int_M\frac{\varepsilon}{2}|\nabla u|^2+\frac{W(u)}{\varepsilon}.
\end{equation*}
A smooth function $u$ is a critical point of $E_\varepsilon$ if it solves the $\varepsilon$\textit{-Allen--Cahn equation}
\begin{equation}\label{eq:AC_eqn_W}
    \varepsilon^2\Delta u=W'(u).
\end{equation}

We shall now discuss the min-max theory associated to critical points of the $\varepsilon$-Allen--Cahn energy, as formulated in \cite{gaspar2018allen,Dey}. Let $X$ be some finite dimensional cubical complex, and let $\pi:\tilde{X}\to X$ be a double cover. 
We shall associate to $\pi$ an (Almgren--Pitts) homotopy class $\Pi$ of $\mathbf{F}$-continuous maps $\Phi:X\to\mathcal{Z}_1(M;\Z_2)$ as follows.
\begin{definition}
Given some finite dimensional cubical complex $X$ and a double cover $\pi:\tilde{X}\to X$, an $\mathbf{F}$-continuous map $\Phi:X\to\mathcal{Z}_1(M;\Z_2)$ is in $\Pi$ if and only if
\begin{equation*}
    \mathrm{ker}(\Phi_*:\pi_1(X)\to\pi_1(\mathcal{Z}_1(M;\Z_2)))=\mathrm{im}(\pi_*).
\end{equation*}
Namely, $\Pi$ is the $\mathcal{F}$-homotopy class of $\mathbf{F}$-continuous maps corresponding to the double cover $\pi:\tilde{X}\to X$.
\end{definition}

Similarly, we can define the following collection $\tilde{\Pi}$ of continuous maps $\tilde{X}\to H^1(M)\setminus\{0\}$ corresponding to the double cover $\pi$, and its $\varepsilon$\textit{-phase transition width} $\mathbf{L}_\varepsilon(\tilde{\Pi})$ as follows. 
\begin{definition}\label{def:AC_class}
Given some finite dimensional cubical complex $X$ and a double cover $\pi:\tilde{X}\to X$, let $\tau:\tilde{X}\to\tilde{X}$ be the unique nontrivial deck transformation corresponding to $\pi$, and let $\tilde{\Pi}$ be the collection of all continuous $\Z_2$-equivariant maps $h:\tilde{X}\to H^1(M)\setminus\{0\}$, in the sense that
\[h(\tau(x))=-h(x)\]
for all $x\in\tilde{X}$. We shall write $\tilde{\Pi}=\Gamma(\tilde{X})$ to highlight the dependence on the double cover $\tilde{X}\to X$, and refer to such a class as a \textit{phase transition} class\footnote{Notice this purely topological notion is independent of $\varepsilon>0$.}.
We define the $\varepsilon$\textit{-phase transition width} corresponding to $\tilde{\Pi}$ by
\begin{equation*}
    \mathbf{L}_\varepsilon(\tilde{\Pi})\coloneqq \inf_{h\in\tilde{\Pi}}\sup_{x\in\tilde{X}}E_\varepsilon[h(x)].
\end{equation*}
\end{definition}
We have the following existence result.
\begin{proposition}[{\cite[Theorem 3.3]{gaspar2018allen}}]\label{prop:AC_minmax_existence}
Let $\tilde{\Pi}$ be as in Definition \ref{def:AC_class}. If 
\[\mathbf{L}_\varepsilon(\tilde{\Pi})<E_\varepsilon[0]=\frac{W(0)}{\varepsilon}\mathrm{Vol}(M,g),\] then there is a min-max critical point $u_\varepsilon\in H^1(M)\setminus\{0\}$ of $E_\varepsilon$ corresponding to $\tilde{\Pi}$. Moreover, $u_\varepsilon\in C^\infty(M)$, $|u_\varepsilon|\leq 1$, and it solves \eqref{eq:AC_eqn_W}. Furthermore, the Morse index of $u_\varepsilon$ as a critical point of $E_\varepsilon$ is bounded above by $\mathrm{dim}(X)$.
\end{proposition}

In order to define the $\varepsilon$-phase transition spectrum, we now just need to find an appropriate sequence of phase transition classes $\tilde{\Pi}$ to which we can apply Proposition \ref{prop:AC_minmax_existence}. This can be done via a construction due to Gaspar--Guaraco \cite{gaspar2018allen}. 
We refer to \cite[Section 1]{Dey} for a description of this construction, which follows a similar notation and approach to one the we have adopted in this section.
The upshot of Gaspar--Guaraco's construction is the existence of a decreasing sequence $\{\mathcal{C}_p\}_{p\in\N}$ of collections of $p$-dimensional cubical complexes $\tilde{X}$, each possessing a free $\Z_2$ action (so that $\tilde{X}$ can be seen as the total space of a double cover $\tilde{X}\to X$), which provides the adequate analogue of the sequence $\{\mathcal{P}_p\}_{p\in\N}$ of Almgren--Pitts $p$-sweepouts in the phase transition setting.

The $\varepsilon$-phase transition spectrum $\{c_{\varepsilon,p}(M,g, W)\}_{p\in\N}$ can then be defined by
\begin{equation*}
    c_{\varepsilon,p}(M,g, W)\coloneqq\inf_{\tilde{X}\in\mathcal{C}_p}\mathbf{L_\varepsilon}(\Gamma(\tilde{X})).
\end{equation*}

Finally, one can define the \textit{phase transition spectrum} $\{c_p(M,g, W)\}_{p\in\N}$ by taking the limit as $\varepsilon\searrow 0$ (this is well defined by \cite{Dey}):
\begin{equation*}
    c_p(M,g,W)\coloneqq h_0^{-1}\lim_{\varepsilon\searrow 0}c_{\varepsilon,p}(M,g,W),
\end{equation*}
where $h_0$ is the squared $L^2$-energy of the heteroclinic solution $\mathbb{H}$ on $\R$, i.e. the unique (up to sign and translations) non-constant finite energy solution of the ODE
\[\mathbb{H}''(t)=W'(\mathbb{H}(t))\]
with $\mathbb{H}(0)=0$.

The key link between the $\varepsilon$-phase transition min-max theory and geodesics on $(M,g)$ lies in an important regularity result due to the combined work of Hutchinson--Tonegawa \cite{hutchinson2000convergence}, and Guaraco \cite{guaraco2018min}.

In order to state it, we need to associate to a solution $u$ of \eqref{eq:AC_eqn_W} a 1-varifold $V_\varepsilon[u]$ by
\begin{equation*}
    V_\varepsilon[u](f)\coloneqq h_0^{-1}\int_{\{|\nabla u|\neq 0\}}\varepsilon|\nabla u|^2 f(x,T_x\{u=u(x)\})
\end{equation*}
for $f\in C^0(G_1(M))$.

\begin{theorem}\label{thm:HT_reg}
Let $\{(u_i,\varepsilon_i)\}_i\subset C^\infty(M)\times (0,+\infty)$, where $\varepsilon_i\to 0$, and each $u_i$ satisfies
\[\varepsilon_i^2\Delta u_i=W'(u_i)\]
on M.
Assume 
\[\limsup_i\sup_M|u_i|\leq c_0<+\infty,\, \limsup_iE_{\varepsilon_i}[u_i]\leq E_0<+\infty.\]
Then, after passing to a subsequence,
\begin{itemize}
    \item $V_{\varepsilon_i}[u_i]\to V^\infty$ for a stationary integral 1-varifold $V^\infty$,
    \item $\|V^\infty\|(M)=h_0^{-1}\lim_i E_{\varepsilon_i}[u_i]$.
\end{itemize}
\end{theorem}

Similarly to the Almgren--Pitts theory, for a nontrivial phase transition class $\tilde{\Pi}$, we can define the (phase transition) \textit{critical set} $\mathbf{C}_{\mathrm{PT}}(\tilde{\Pi})$.

\begin{definition}
Given a nontrivial $\tilde{\Pi}$ as in Definition \ref{def:AC_class}, we define the (phase transition) critical set $\mathbf{C}_{\mathrm{PT}}(\tilde{\Pi})$ to be the set of all stationary integral 1-varifolds such that
\[V=\lim_i V_{\varepsilon_i}[u_i]\]
for some sequence $\{(u_i,\varepsilon_i)\}_i$ as in Theorem \ref{thm:HT_reg}, corresponding to $\tilde{\Pi}$.
\end{definition}

\begin{remark}
By Proposition \ref{prop:AC_minmax_existence} and Theorem \ref{thm:HT_reg}, $\mathbf{C}_{\mathrm{PT}}(\tilde{\Pi})\neq\emptyset$ for a nontrivial phase transition class $\tilde{\Pi}$.
\end{remark}

\subsection{Comparison between the min-max theories}
The following key result links the two min-max theories.
\begin{theorem}[{\cite[Theorem 6.1]{gaspar2018allen}, \cite[Theorem 1.2]{Dey}}]
The $\varepsilon$-phase transition widths and Almgren--Pitts widths are related by
\begin{equation*}
    \mathbf{L}_\mathrm{PT}(\tilde{\Pi)}\coloneqq h_0^{-1}\lim_{\varepsilon\searrow 0}\mathbf{L}_\varepsilon(\tilde{\Pi})=\mathbf{L}_\mathrm{AP}(\Pi),
\end{equation*}
where $\tilde{\Pi}=\Gamma(\tilde{X})$ for the double cover $\pi:\tilde{X}\to X$ corresponding to the Almgren--Pitts homotopy class $\Pi$.
\end{theorem}
Note that this implies that the width $\mathbf{L}_\mathrm{PT}$ of a phase transition class $\tilde{\Pi}$ is independent of the specific form of the double-well potential $W$.
Therefore,
\begin{corollary}
The phase transition spectrum $\{c_p(M,g,W)\}_{p\in\N}$ coincides with the volume spectrum $\{\omega_p(M,g)\}_{p\in\N}$, i.e.
\begin{equation*}
    c_p(M,g,W)=\omega_p(M,g)
\end{equation*}
for all $\in\N$ (independently of the specific form of the double-well potential $W$).
\end{corollary}
\noindent Moreover, Dey \cite{Dey} showed that 
\begin{proposition}\cite[Theorem 1.4]{Dey}
$\mathbf{C}_\mathrm{PT}(\tilde{\Pi})\subset \mathbf{C}_\mathrm{AP}(\Pi)$.
\end{proposition}

Up until this point, all the results we have described in this section are independent of the specific form of the double-well potential $W$. However, it turns out that, in the case of ambient dimension $n+1=2$, a clever choice of double-well potential enables one to make use of a remarkable result of Liu--Wei \cite{liu2022sine-gordon}, which classifies entire phase transitions on $\R^2$ that are regular at infinity, provided one uses the \textit{sine-Gordon} double well-potential. We refer to \cite[Section 3]{chodosh2023p} for more details. Nevertheless, we shall briefly record the geometric consequences of Liu--Wei's classification, due to Chodosh--Mantoulidis in \cite{chodosh2023p}. The statement below is a more explicit description of \Cref{CMGeodesicThm}.

\begin{theorem}[{\cite[Theorem 3.1]{chodosh2023p}}]\label{thm:sine-gordon}
Let $W$ be the sine-Gordon double-well potential defined by
\begin{equation}\label{eq:sine-Gordon}
    W(t)\coloneqq\frac{1+\cos(\pi t)}{\pi^2}.
\end{equation}
Let $\{(u_i,\varepsilon_i)\}_i\subset C^\infty(M)\times (0,+\infty)$, where $\varepsilon_i\to 0$, and each $u_i$ satisfies
\[\varepsilon_i^2\Delta u_i=W'(u_i)\]
on $M$, where $W$ is given by \eqref{eq:sine-Gordon}. Assume that
\[\limsup_i\sup_M|u_i|\leq c_0<+\infty,\, \limsup_iE_{\varepsilon_i}[u_i]\leq E_0<+\infty,\]
and that the Morse index of $u_i$ as a critical point of $E_{\varepsilon_i}$ is uniformly bounded, i.e.
\begin{equation*}
    \limsup_i\mathrm{index}_{E_{\varepsilon_i}}(u_i)\leq I_0<+\infty.
\end{equation*}
Then, after passing to a subsequence, the corresponding varifolds $V_{\varepsilon_i}[u_i]$ converge to a stationary integral 1-varifold $V$ such that
\begin{equation*}
    V=\sum_{j=1}^N m_j|\sigma_j|,
\end{equation*}
where $\sigma_1, \dots, \sigma_N$ are closed immersed geodesics, and $m_1,\dots, m_N$ are positive integers.
\end{theorem}
\noindent From this result, Theorem \ref{CMGeodesicThm} follows directly.

From now on, we shall always assume our double-well potential is the sine-Gordon potential
\[W(t) = \frac{1+\cos(\pi t)}{\pi^2},\]
and we shall restrict our attention to elements in $\mathbf{C}_\mathrm{PT}(\tilde{\Pi})\subset\mathbf{C}_\mathrm{AP}(\Pi)$ which are supported on a union of closed immersed geodesics as in Theorem \ref{thm:sine-gordon}.

\section{Preliminaries} \label{prelims}

Let $(M^2, g)$ be a closed Riemannian surface.

\subsection{Notations and terminology}
Let $\gamma : S^1 \to M$ be a $C^2$ immersed parametrized loop in $M$. We write
\[ \mathrm{length}(\gamma) = \int_{S^1} |\gamma'(\theta)|\ d\theta. \]
We write $T(s) := \gamma'(s)/|\gamma'(s)|$ as the unit tangent vector field along $\gamma$. The curvature of $\gamma$ is the vector field $\kappa = -\nabla_TT$. $\gamma$ is a geodesic if $\kappa = 0$.

\begin{definition}
    $\gamma$ is \emph{primitive} if there is no decomposition $S^1 = P_1\sqcup P_2$ by disjoint connected subsets so that $\gamma(P_1) = \gamma(P_2) = \gamma(S^1)$.
\end{definition}

In colloquial terms, $\gamma$ is primitive if the parametrization traverses its image once (and not multiple times). We let
\begin{itemize}
    \item $\mathcal{G}$ denote the set of finite collections of immersed geodesic loops in $(M, g)$, each parametrized by constant speed on the circle of length $2\pi$.
    \item $\mathcal{G}_{\mathrm{prim}} \subset \mathcal{G}$ denote the set of $\Gamma \in \mathcal{G}$ so that each $\gamma \in \Gamma$ is primitive and for any $\gamma^1,\ \gamma^2 \in \Gamma$ with $\gamma^1 \neq \gamma^2$, $\gamma^1(S^1) \neq \gamma^2(S^1)$.
\end{itemize}

\begin{remark}\label{rem:prim}
    We can associate to any $\Gamma \in \mathcal{G}$ some $\tilde{\Gamma} \in \mathcal{G}_{\mathrm{prim}}$ by deleting parametrized loops with redundant images and taking a primitive parametrization of the image of each remaining loop. Moreover, $\tilde{\Gamma}$ is canonical up to reparametrization of each $\tilde{\gamma} \in \tilde{\Gamma}$. We call such a $\tilde{\Gamma}$ a \emph{primitive representative} of $\Gamma$.
\end{remark}

\begin{definition}
    Given $\Gamma\in\mathcal{G}_{\mathrm{prim}}$, for $x \in M$, define 
    %
    \begin{equation} \label{orderDef}
        \text{ord}_{\Gamma}(x):= \sum_{\gamma \in \Gamma}\#\gamma^{-1}(x)
    \end{equation}
    For general $\Gamma \in G$, we define $\text{ord}_{\Gamma}(x): = \text{ord}_{\tilde{\Gamma}}(x)$ for the canonically associated $\tilde{\Gamma}$ as in Remark \ref{rem:prim}.
\end{definition}

\begin{definition} \label{SelfIntDef}
    $\Gamma \in \mathcal{G}_{\mathrm{prim}}$ has a \emph{vertex} at $x \in M$, written $x \in \mathrm{Vert}(\Gamma)$, if $\text{ord}_{\Gamma}(x) > 1$.
    %
\end{definition}

\begin{remark}
    The set of vertices $\mathrm{Vert}(\Gamma)$ is finite for any $\Gamma \in \mathcal{G}_{\mathrm{prim}}$.
\end{remark}

\begin{definition} 
$\Gamma \in \mathcal{G}_{\mathrm{prim}}$ has an \emph{order $2$} vertex at $x\in M$ if $\ord_{\Gamma}(x) = 2$. 
\end{definition}

We let
\begin{itemize}
    \item $\mathcal{G}_{+} \subset \mathcal{G}_{\mathrm{prim}}$ denote the set of $\Gamma \in \mathcal{G}_{\mathrm{prim}}$ so that every vertex of $\Gamma$ is order $2$.
\end{itemize}

\subsection{Compactness under length bounds}
%

\begin{definition}
    $\{\Gamma_i\}_{i\in\N} \subset \mathcal{G}$ converges \emph{in $C^k$} to $\Gamma = \{\gamma^1, \hdots, \gamma^N\} \in \mathcal{G}$ if for all $i$ sufficiently large we have $\Gamma_i = \{\gamma_i^1, \hdots, \gamma_i^N\}$ so that $\{\gamma_i^j\}_i$ converges to $\gamma^j$ in the $C^k$-graphical sense for immersions. We say this convergence is \emph{smooth} if it holds for all $k \geq 0$.
\end{definition}

For $L > 0$, let $\mathcal{G}^L$ denote the set of $\Gamma \in \mathcal{G}$ with
\[ \mathrm{length}(\Gamma) := \sum_{\gamma \in \Gamma} \mathrm{length}(\gamma) \leq L. \]
We also let $\mathcal{G}^L_{\mathrm{prim}} = \mathcal{G}^L \cap \mathcal{G}_{\mathrm{prim}}$.

\begin{theorem}\label{thm:compactness}
 Let $g$ be a $C^{k, \alpha}$ metric with $k \geq 2$. Then $\mathcal{G}^L$ is sequentially compact under $C^{\ell}$ convergence for any $1 \leq \ell \leq k$. 
\end{theorem}
\begin{proof}
First consider the case of $\Gamma_i = \{\gamma_i\}$ (i.e.\ $\Gamma_i$ consists of one parametrized loop).

\emph{Convergence as maps}: By assumption, we have $|\gamma_i'(s)| = L(\gamma_i)/2\pi \leq L/2\pi$. We isometrically embed $(M, g)$ in $\R^L$. Since $\gamma_i$ is a geodesic, we have $\frac{D}{dt}\frac{d\gamma_i}{dt} = 0$, and therefore $\frac{D^{\ell}}{dt^{\ell}}\frac{d\gamma_i}{dt} = 0$ holds with $C^{\alpha}$ coefficients for all $1 \leq \ell \leq k$. Since $M$ is compact and smooth, we have uniform $C^{\ell + 1, \alpha}$ bounds for $\{\gamma_i\}$ as maps to $\R^L$. By Arzelà-Ascoli, there is a subsequence (not relabeled) that converges in $C^{\ell + 1}$ as maps from $S^1 \to M \subset \R^L$ to a constant speed parametrized geodesic $\gamma$ in $M$. 

\emph{Convergence as graphs}: Since $M$ is compact and $\gamma$ is a closed geodesic, there is some $\eps > 0$ so that the normal exponential map along $\gamma$ gives a diffeomorphism from $N_{\eps}\gamma$ (the vectors of length at most $\eps$ in the normal bundle of $\gamma$) to a small tubular neighborhood $U$ of $\gamma$ in $M$. By $C^0$ convergence, $\gamma_i(S^1) \subset U$ for all $i$ sufficiently large. By $C^{\ell + 1}$ convergence as maps, we can find $C^{\ell}$ sections $S_i : S^1 \to N_{\eps}\gamma$ parametrizing $\gamma_i$ under the normal exponential map, converging smoothly to the zero section as $i \to \infty$, which gives $C^{\ell}$ graphical convergence.

%

Generalizing to any element $\Gamma \in \mathcal{G}^L$, note that any closed geodesic will have length at least the injectivity radius of $(M, g)$, which is positive. Thus any $\Gamma \in \mathcal{G}^L$ has a bounded number of geodesic loops, and we can repeat the one element construction a finite number of times. 
\end{proof}
We record similar lemmas for varying metrics, which follow from the fact that, if $g_i \xrightarrow{C^2} g_0$, then the coefficients of the $g_i$-geodesic equation converge to those of the $g_0$-geodesic equation.
\begin{restatable}{lemm}{lem:compact_var_met}\label{lem:compact_var_met}
Suppose $g_i \xrightarrow{C^2} g_0$ is a converging sequence of metrics and $\{\gamma_i\}$ is a sequence of geodesics with respect to $g_i$ such that $\mathrm{length}_{g_i}(\gamma_i) \leq L$. Then up to a subsequence, we have graphical convergence of $\gamma_i \xrightarrow{C^2} \gamma_0$, a geodesic with respect to $g_0$.
\end{restatable}
\begin{restatable}{lemm}{lem:Jacobi_fields}\label{lem:Jacobi_fields}
Suppose there exist pairs $\{(\gamma_i^1, \gamma_i^2)\}$ of geodesics with respect to $g_i$ such that $\gamma_i^1, \gamma_i^2 \xrightarrow{C^{2}} \gamma_0$ (graphically), $g_i \xrightarrow{C^2} g_0$, and the graphical parametrizations of $\gamma_i^1$ and $\gamma_i^2$ are distinct. If $\gamma_0$ is two-sided, then it admits a nontrivial Jacobi field with respect to $g_0$. If it is one-sided, then its double cover admits a nontrivial Jacobi field. 
\end{restatable}
\begin{proof}
For all large enough $i$, we can represent $\gamma^1_i$ and $\gamma^2_i$ as graphs over $\gamma_0$ as follows
\begin{align*}
\gamma_i^1 &= \exp_{\gamma_0, g_i}(u_i^1) \\
\gamma_i^2 &= \exp_{\gamma_0, g_i}(u_i^2),
\end{align*}
where we emphasize that the exponential map is being computed with respect to $g_i$. We also emphasize that $u_i^1, u_i^2$ are sections of the normal bundle of $\gamma_0$, so that locally (but not globally if $\gamma_0$ is unoriented), we can represent $u_i^k = f_i^k(s) \nu(s)$ where $f_i^k: S^1 \to \R$ and $\nu$ is a local choice of normal. Now define 
\[
\phi_i = \frac{u_i^1 - u_i^2}{|| u_i^1 - u_i^2||_{C^0}}
\]
and consider
\begin{equation*}
0 = \frac{H_{\gamma_i^1} - H_{\gamma_i^2}}{||u_i^1 - u_i^2||_{C^0}}
\end{equation*}
so that
\begin{equation*}
0 = J_{\gamma_0, g_i}(\phi_i) + O(||u_i^1 - u_i^2||_{C^0}),
\end{equation*}
where $J_{\gamma_0, g_i}$ denotes the Jacobi operator of the normal bundle of $\gamma_0$ with respect to the metric $g_i$. When $\gamma_0$ is two-sided, the locally defined function $\phi_i$ patches together to form a global map $\phi_i: \gamma_0 \to \R$. We apply Arzelà--Ascoli to the normalized functions $\{\phi_i\}$, we get convergence $\phi_i \to \phi_{\infty}$, and $\phi_\infty$ is a Jacobi field on $\gamma_0$ (with respect to $g_0$) such that $||\phi_{\infty}||_{C^0} = 1$. When $\gamma_0$ is one-sided, we pass to the double cover and repeat the construction.
\end{proof}
We note that the above also handles the case in which a primitive geodesic converges with multiplicity: 
\begin{restatable}{corr}{GeoMultCor} \label{GeoMultCor}
Suppose $\{\gamma_i\}$ is a sequence of primitive geodesics with respect to $g_i$ such that $g_i \xrightarrow{C^2} g_0$ and $\gamma_i \xrightarrow{C^2} m \cdot \gamma_0$ (graphically), where $\gamma_0$ is primitive and $m \cdot \gamma_0$ denotes $\gamma_0$ traversed with multiplicity $m \geq 2$. If $\gamma_0$ is two-sided, then it admits a  nontrivial Jacobi field. If $\gamma_0$ is one-sided, then its double cover admits a nontrivial Jacobi field. 
\end{restatable}
\begin{proof}
Let $p: [0, 2 \pi m] \to \gamma_0$ be an $m$-fold parameterization of $\gamma_0$. By nature of (local) graphical convergence, we know that for $i$ sufficiently large, we can write
\begin{equation} \label{LocalGraphical}
\gamma_i(s) = \exp_{\gamma_0}(u_i(s)),
\end{equation}
locally in $s$, where $u_i$ is a section of $N \gamma_0$.
For $i$ sufficiently large, we have $||u_i||_{C^2} \leq \eps$ for $\eps > 0$ arbitrarily small.  \nl 
\indent We now want to extend $u_i$ the domain of $u_i$ from $[s - \delta, s + \delta]$ to a map on all of $[0, 2 \pi m]$. We achieve this by taking a finite covering of $p \cdot \gamma_0$ by intervals of the form $\{(s_k - \delta, s_k + \delta) \times (- \delta, \delta)_t\}$ where $\delta$ can be chosen uniformly given $\gamma_0$ fixed. \nl 
\indent First assume that $\gamma_0$ is 2-sided. Given a unique choice of smooth graphical representation $u_i \nu$ (where $\nu$ is a global unit vector) restricted to $(s_1 - \delta, s_1 + \delta)$, this determines a choice of $u_i$ in $(s_2 - \delta, s_2 + \delta)$ by looking at the overlap region of $(s_1 - \delta, s_1 + \delta) \cap (s_2 - \delta, s_2 + \delta)$ and then choosing the unique continuous extension into the rest of $(s_2 - \delta, s_2 + \delta)$. Repeat this inducetively for all $(s_k - \delta, s_k + \delta)$, noting that $\gamma_0$ being 2-sided will allow us to enforce $u_i(2 \pi m) = u_i(0)$. \nl 
\indent If $\gamma_0$ is 1-sided, then we pass to a double cover, $p \cdot \overline{\gamma}_0$ where $\overline{\gamma}_0$ is two-sided. Similarly, we can pass to $\overline{\gamma}_i$, the two sided double cover of $\gamma_i$, and repeat the proof. This means that \eqref{LocalGraphical} holds modulo switching $p: [0, 2\pi m] \to \gamma_0$ for a parameterization of the double cover, $\overline{p}:[ 0, 2\pi m] \to \overline{\gamma}_0$, and defining $\overline{u}_i$ with respect to $\overline{\gamma}_i$ converging to $\overline{\gamma}_0$. \nl 
%
\indent In the two-sided case, we note that since $p(s) = p(s + 2\pi)$ for all $s$ (mod $2\pi m$), we can replicate the proof of Lemma \ref{lem:Jacobi_fields} but with these explicit parametrizations. Consider 
\[
\phi_i(s) = \frac{u_i(s) - u_i(s + 2\pi)}{||u_i(s) - u_i(s + 2 \pi)||_{C^0}}.
\]
Note that the denominator is non-zero, else $u_i(s) = u_i(s + 2 \pi)$ for all $s$, a contradiction to $\gamma_i$ being primitive. Then, as in Lemma \ref{lem:Jacobi_fields}, we have 
\begin{equation*}
0 = \frac{H_{\gamma_i(s)} - H_{\gamma_i(s + 2\pi)}}{||u_i(s) - u_i(s + 2 \pi)||_{C^0}},
\end{equation*}
i.e.
\begin{align*}
0 &= J_{\gamma_0, g_i}(\phi_i(s)) + O(||u_i(s) - u_i(s + 2\pi||_{C^0}) \\
&= J_{\gamma_0, g_i}(\phi_i(s)) + O(||u_i||_{C^0})
\end{align*}
Applying Arzelà-Ascoli to $\{\phi_i\}$, we get a non-zero Jacobi field as before. The one-sided case follows similarly using $\overline{p}(s)$, $\overline{\gamma}_0$, and $\overline{u}_i$ instead.
\end{proof}
\begin{figure}[ht!]
    \centering
    \includegraphics[scale=0.4]{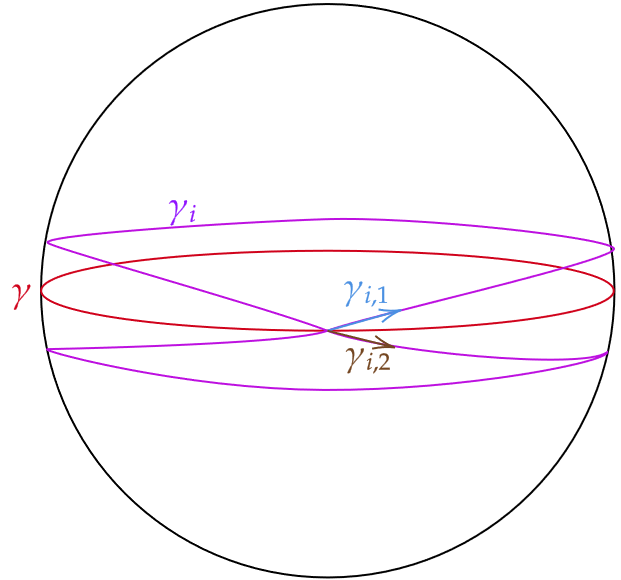}
    \caption{Corollary \ref{GeoMultCor} visualized}
    \label{fig:geomult}
\end{figure}

\subsection{First and second variation of length}
Let $\gamma : S^1 \to M$ be a $C^2$ immersed loop in $M$. Let $X$ be a vector field on $M$, and let $\Phi_t$ be the flow generated by $X$. We recall the formulas for the first and second variation of length.

\begin{proposition}\label{prop:variation_formulas}
    \[ \frac{d}{dt}\Big|_{t=0} \mathrm{length}(\Phi_t(\gamma)) = \int_{\gamma} \langle X, \kappa\rangle\ d\gamma, \]
    \[ \frac{d^2}{dt^2}\Big|_{t=0} \mathrm{length}(\Phi_t(\gamma)) = \int_{\gamma} (|\nabla_T^{\perp}X|^2 + \langle \nabla_XX, \kappa\rangle - |X^{\perp}|^2K_M)\ d\gamma. \]
    Moreover, if $\gamma$ is a geodesic, then
    \[ \frac{d}{dt}\Big|_{t=0} \mathrm{length}(\Phi_t(\gamma)) = 0,  \]
    \[ \frac{d^2}{dt^2}\Big|_{t=0} \mathrm{length}(\Phi_t(\gamma)) = \int_{\gamma} (|\nabla_T^{\perp}(X^{\perp})|^2 - |X^{\perp}|^2K_M)\ d\gamma. \]
\end{proposition}
\begin{proof}
    For a smoothly immersed closed $k$-dimensional submanifold $\Sigma$ in an $n$-dimensional Riemannian manifold $N$, the first and second variation formulas (see \cite{CMbook}) take the form (where $\{e_i\}_{i=1}^k$ is a local orthonormal frame for $\Sigma$) 
    \[ \delta \Sigma(X) = \int_{\Sigma} \mathrm{div}_{\Sigma}(X)\ d\Sigma, \]
    \begin{align*}
        \delta^2 \Sigma(X)
        & = \int_{\Sigma} \sum_{i=1}^k |D_{e_i}^{\perp}X|^2 + \mathrm{div}_{\Sigma}(\nabla_X X) - \sum_{i=1}^k R^N(e_i, X, X, e_i)\\
        & \hspace{1cm} + \left(\sum_{i=1}^k\langle \nabla_{e_i}X, e_i\rangle\right)^2 - \sum_{i,j=1}^k \langle \nabla_{e_i}X, e_j\rangle \langle \nabla_{e_j}X, e_i\rangle\ d\Sigma.
    \end{align*}
    
    First consider the first variation formula in our setting. Since $\gamma$ is closed, we have
    \[ \int_{\gamma} \mathrm{div}_{\gamma}(X^T)\ d\gamma = 0. \]
    Moreover, since $\kappa = -\nabla_TT$ is perpendicular to $\gamma$, we have
    \[ \mathrm{div}_{\gamma}(X^{\perp}) = \langle \nabla_T (X^{\perp}), T\rangle = -\langle X, \nabla_TT\rangle = \langle X, \kappa\rangle, \]
    so the first variation formula follows.

    Now consider the second variation formula. Write $X^T = \phi T$ along $\gamma$. We have
    \[ \nabla_T^{\perp}(X^T) = (T(\phi)T + \phi\nabla_TT)^{\perp} = -\phi\kappa, \]
    so when $\gamma$ is a geodesic we have
    \[ |\nabla_T^{\perp}X|^2 = |\nabla_T^{\perp}(X^{\perp})|^2. \]
    Since $\gamma$ is closed, we have
    \[ \int_{\gamma} \mathrm{div}_{\gamma}(\nabla_X^TX)\ d\gamma = 0. \]
    Moreover, since $\kappa = -\nabla_TT$ is perpendicular to $\gamma$, we have
    \[ \mathrm{div}(\nabla_X^{\perp}X) = \langle \nabla_T(\nabla_X^{\perp}X), T\rangle = - \langle \nabla_XX, \nabla_TT\rangle = \langle \nabla_XX, \kappa\rangle. \]
    For the curvature term, we observe that $R^M(T, X^T, X^T, T) = 0$ by the symmetries of the curvature tensor, and $R^M(T, X^{\perp}, X^{\perp}, T) = |X^{\perp}|^2K_M$. 
    Since $k=1$, the last two terms in the general second variation formula cancel, so the proposition follows.
\end{proof}

\begin{remark}
An essential consequence of Proposition \ref{prop:variation_formulas} is that the second variation of length for a geodesic only depends on the normal projection of $X$ along $\gamma$.  
\end{remark}

\section{Generic structure of self intersections }\label{GenericInterSec}
The goal of this section is to prove 
\begin{theorem} \label{GenSelfIntFull}
The set of metrics for which $\mathcal{G}_{\mathrm{prim}} = \mathcal{G}_{+}$ is $C^{k}$-generic for any $k \geq 3$ in the Baire sense.
\end{theorem}
\noindent Note that Theorem \ref{GenSelfIntFull} immediately gives \Cref{GenSelfIntThm}. See Figure \ref{fig:genInt} for a visualization: 
%
\begin{figure}[H]
    \centering
    \includegraphics[scale=0.5]{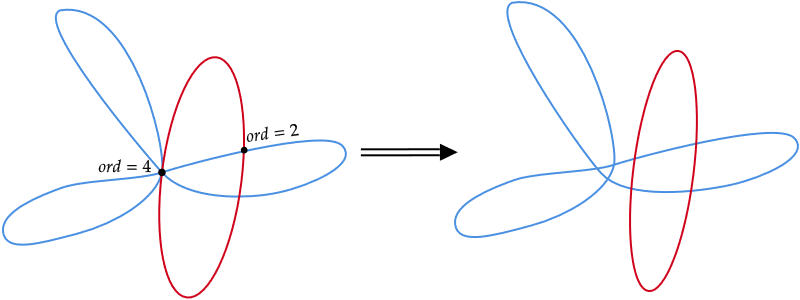}
    \caption{L: Non-generic metric, R: generic metric}
    \label{fig:genInt}
\end{figure}
%
%
We claim there exists a small perturbation $g_{\eps}$ of $g$ such that any $\Gamma \in \mathcal{G}^L_{\mathrm{prim}}(g_{\eps})$ 
has self-intersections of order at most $2$, and any $\Gamma \in \mathcal{G}^L(g_{\eps})$ does not admit any nontrivial Jacobi fields.

Recall a metric $g$ on $M$ is said to be \textit{bumpy} if no nontrivial closed immersed $g$-geodesic (or any of its finite covers) admits a nontrivial normal Jacobi field. The set of bumpy metrics on $M$ is $C^k$-generic in the space of all $C^k$ Riemannian metrics on $M$ (see \cite{abraham1970bumpy},\cite{white1991space}, \cite{white2017bumpy}) for any $k \geq 3$.

%
\begin{restatable}{prop}{DegFourProp} \label{DegFourProp}
Suppose $(M^2, g)$ is bumpy and $g$ is $C^k$. Then for any $\eps > 0$, $L>0$, there exists another metric, $g_{\eps, L}$ with
\[
||g_{\eps, L} - g||_{C^k} \leq \eps
\]
such that 
\begin{itemize}
    \item every $\Gamma \in \mathcal{G}^L_{\mathrm{prim}}(g_{\eps, L})$ has vertices with order at most 2,
    \item every $\Gamma \in \mathcal{G}^L(g_{\eps, L})$ does not admit any nontrivial Jacobi fields.
\end{itemize}
\end{restatable}
%
%

\begin{proof} 
Let $S(L, g)$ denote the set of connected primitive closed immersed geodesics of length at most $L$. For any bumpy metric $g$, $S(L, g)$ is a finite set by Lemma \ref{lem:Jacobi_fields}. The idea is to perturb each element of $S(L, g)$ in a way that creates intersections of order 2, and also avoids creating higher order intersections for any element of $S(2L, g) \setminus S(L, g)$. 

Without loss of generality, we can assume that $g$ is smooth, since otherwise we can choose a smooth bumpy metric $g^*$ such that $||g^* - g||_{C^k} < \eps/2$. Then, if we can find a perturbation, $g_{\eps, L}$, of $g^*$ such that $||g_{\eps, L} - g^*||_{C^k} < \eps/2$ and $g_{\eps,L}$ has the desired properties, we'll have proved the proposition.

Define 
\[
\Gamma = \{\gamma \mid \gamma \in S(2L, g)\},
\]
which is in $\mathcal{G}^K(g)$ for some $K$ large but finite depending on $g$ and $L$. Let $\{v_{i}\}$ denote the set of vertices in $\Gamma$ with order $\geq 3$ self-intersection and $\{w_i\}$ the set of vertices with order $2$ self-intersection. For each $v_{i}$, we will demonstrate how to decrease the order of the vertex by introducing vertices with order $2$ self-intersection. The construction then proceeds inductively to decrease the order of all vertices in $\{v_i\}$, at the cost of creating more vertices with order $2$. Suppose $v_{i}$ has an order $d\geq 3$ self-intersection. Let $r_{inj}$ denote the injectivity radius of $M$. Choose a small geodesic neighborhood $U_{i} \ni v_{i}$ with diameter less than $r_{inj}$ and such that $U_i \cap \{w_j\} = \emptyset$, along with a chart map 
\[
\varphi_{i}: B_1(0) \to U_{i}
\]
such that $||\varphi_{i}^*(g) - g_{euc}||_{C^{k+2}} \leq \delta$. Note that for any $\delta > 0$, such a $U_{i}$ exists by choosing it sufficiently small. Since $v_{i}$ has order $d$, $\Gamma$ in $U_{i}$ consists of an intersection of $d$ geodesic segments, $\{\gamma_{j}\}_{j = 1}^d$. 
\begin{restatable}{lemm}{confChangeLem} \label{confChangeLem}
There exists a conformal change of metric, $g_{\eps, i}$ such that 
\begin{itemize}
\item $g_{\eps, i} = g$ on $M \backslash U_{i}$, $||g_{\eps,i} - g||_{C^k} \to 0$ on $U_{i}$ as $\varepsilon\to 0$;
\item the curves $\{\gamma_j\}_{j = 2}^d$ are geodesics with respect to $g_{\eps, i}$ on all of $M$;
\item there exists a curve $\gamma_1^*$ such that
\begin{itemize}
    \item $\gamma_1^*$ is a geodesic with respect to. $g_{\eps,i}$,
    \item $\gamma_1^* = \gamma_1$ outside of $U_{i}$,
    \item $v_{i} \not \in \gamma_1^*$,
    \item $\gamma_1^*$ is graphical over $\gamma_1$, converging to $\gamma_1$ as $\eps \to 0$.
\end{itemize}
\end{itemize}
\end{restatable}
\begin{proof} Consider a vertex, $v_i$, with order greater than or equal to $2$ and let $U_i$ be a small neighborhood about $v_i$ given by the exponential map with radius less than the injectivity radius of $M$. We now take a coordinate chart at $v_i$. Identify a subset of $U_{i}$ with $B_{3r_i/2}(0)\subset \R^2$ via the exponential map based at $v_i$ and the curves $\{\gamma_j\}$ with their images under $\varphi_{i}$, which will be a union of straight lines for a sufficiently small choice of $U_i$. Rescale the chart map $B_{3r_i/2}(0) \to B_{3/2}(0)$ via $\tilde{\varphi}_i = r_i^{-1} \varphi_i$ along with the metric $g_i = r_i^{-2} \tilde{\varphi}_i^*(g)$ so that for $r_i$ sufficiently small
\[
||g_i - g_{euc}||_{C^{k+2}} \leq \delta
\]
We can do this for any $\delta > 0$ small. Now consider
\[
\{P,Q\} := \gamma_1 \cap \partial B_1(0), \qquad \{p,q\} = \gamma_1 \cap \partial B_{1/2}(0)
\]
such that $\text{dist}(p,P) < \text{dist}(p,Q)$ (see Figure \ref{fig:deg4}). 
\begin{figure}[ht!]
    \centering
    \includegraphics[scale=0.4]{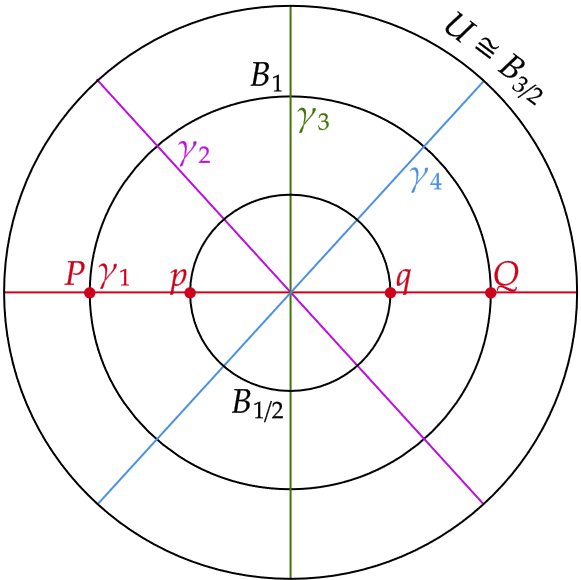}
    \caption{Local chart near self-intersection}
    \label{fig:deg4}
\end{figure}
%
%
Without loss of generality, one can rotate the chart map on $B_{3/2}(0) = \{x^2 + y^2 < 9/4\}$ so that $\gamma_{1}'(p) = \alpha \p_x$ with $\alpha > 0$. Recall that $\varphi_{i}(0) = v_{i}$. Consider the families of points $p_{t} = p + t \p_y$ 
for $t \in [-\eta, \eta]$ with $\eta$ sufficiently small that $p_{t} \in U_{i}$. Consider the uniquely defined geodesic 
\begin{equation} \label{sigmatpath}
\sigma_{p_t, q}: p_t \rightarrow q
\end{equation}
(see Figure \ref{fig:modcurve}). Since $p_t \neq p$ for any $t \neq 0$, we note that $\sigma_{p_t, q} \cap \gamma_1 = \{q\}$, as the existence of any other point of intersection would contradict the fact that $U_{i}$ was chosen with diameter less than the injectivity radius. In particular, $v_{i} \not \in \sigma_{p_t, q}$. Consider the restriction of $\gamma$ from $\gamma(s_0) = P \to p = \gamma(s_1)$ and construct a smooth curve (not necessarily geodesic) $\rho: P \to p_t$ graphical over $\gamma$, i.e.
\begin{align*}
\rho(s) &:= \exp_{\gamma, g}(u_t(s) \nu(s)) \\
u_t(s_0) &:= 0 \\
u_t(s_1) &= t
\end{align*}
so that $\rho(s_0) = P$ and $\rho(s_1) = p_t$. We further enforce that $\rho$ agrees with $\gamma_1$ to second order at $P$, i.e. 
\[
u_t'(s_0) = u_t''(s_0) = 0
\]
We can similarly enforce that $\rho$ agrees with $\sigma_t$ to second order at $p_t$: let $\sigma_{p_t,q}: [s_1, s_2] \to M$ be represented graphically over $\gamma_1$ as 
\begin{align*}
\sigma_{p_t,q}(s) &= \exp_{\gamma_1, g}(v_t(s) \nu(s)) \\
v_t(s_1) &= t \\
v_t(s_2) &= 0 \\
\nabla_{\dot{\sigma}} \dot{\sigma} &= 0
\end{align*}
where $\dot{\sigma} := \frac{d}{ds} \sigma_{p_t, q}(s)$. Then we enforce
\begin{align*}
u_t'(s_1) &= v_t'(s_1) \\
u_t''(s_1) &= v_t''(s_1)
\end{align*}
Finally, we construct a smooth curve $\tau: q \to Q$ which agrees with $\sigma_{p_t, q}$ to second order at $q = \tau(a_0)$ and with the original curve $\gamma_1$ to second order at $Q = \tau(a_1)$, i.e.
\begin{align*}
\tau(a) &:= \exp_{\gamma, g}(b_t(a) \nu(a)) \\
b_t(a_0) &:= 0 \\
b_t'(a_0) &= \sigma_{p_t,q}'(s_2) \\
b_t''(a_0) &= \sigma_{p_t,q}''(s_2) \\
b_t(a_1) &= 0 \\
b_t'(a_1) &= 0 \\
b_t''(a_1) &= 0 
\end{align*}
Note that we can achieve all of these conditions while also enforcing 
\[
||u_t||_{C^{k+2}(s), g_i} = o(1), \qquad ||b_t||_{C^{k+2}(s), g_i} = o(1)
\]
as $t \to 0$. As an example, let $\chi_r(s)$ denote a bump function which is $1$ on $[-r,r]$ and is $0$ outside of $[-2r, 2r]$. Then define 
\begin{align*}
u_t&: [s_0, s_1] \to \R \\
\alpha_0 &= s_1 - s_0 \\
u_t(s) &= \chi_{\alpha_0/4}(s - s_1)[t + (s - s_1) v_t'(s_1) + (s-s_1)^2 v_t''(s_1)] 
\end{align*}
which is clearly smooth and satisfies 
\[
||u_t||_{C^{k+2}, g_i} \leq C(\alpha_0, q) \cdot [|t| + |v_t'(s_1)| + |v_t''(s_1)|]
\]
%
Note that $|t|, |v_t'(s_1)|, |v_t''(s_1)|$ can be made arbitrarily small (independent of $\alpha_0$) by sending $t \to 0$ since the geodesics starting at $p_t$ and ending at $q$ vary smoothly in $t$ as $t \to 0^+$. Thus, with $r_i, \delta, \alpha_0$ fixed, we see that
\[
||u_t||_{C^{k+2}} = o_t(1)
\]
A similar argument works for $b_t$. 
%
%
%
Now consider the conglomerate curve (see Figure \ref{fig:modcurve})
\[
\gamma_1^* = \begin{cases}
	\gamma_1 & M^2 \backslash U_{i,j} \\
	\rho & p \rightarrow p_t \\
	\sigma_{p_t, q} & p_t \rightarrow q \\
    \tau & q \rightarrow Q
	\end{cases}
\]
Since $\{\gamma_j \cap U_{i}\}_{j = 2}^d$  all intersect at $v_{i}$, the fact that $v_{i} \not \in \gamma_1^*$, means that $(\gamma_1^* \cap U_{i})$ intersects each $(\gamma_j \cap U_{i})$ in at most one distinct point. This follows because by choosing $t$ small, we can guarantee that 
\[
j \geq 2 \implies \gamma_1^* \cap \gamma_j \cap (U \backslash B_{1/2}(0)) = \emptyset
\]
since this is true for $\gamma_1$ and $\gamma^*$ is an $o_t(1)$ graphical perturbation of $\gamma_1$. We also know that 
\[
j \geq 2 \implies |\gamma_1^* \cap \gamma_j \cap \overline{B}_{1/2}(0)| = 1
\]
since $\gamma_1^*$ and $\gamma_j$ are geodesics with respect to $g$ in $B_{1/2}(0)$ and more than $1$ point of intersection would contradict the choice of $r_i$ being less than the injectivity radius. 

\begin{figure}[ht!]
    \centering
    \includegraphics[scale=0.4]{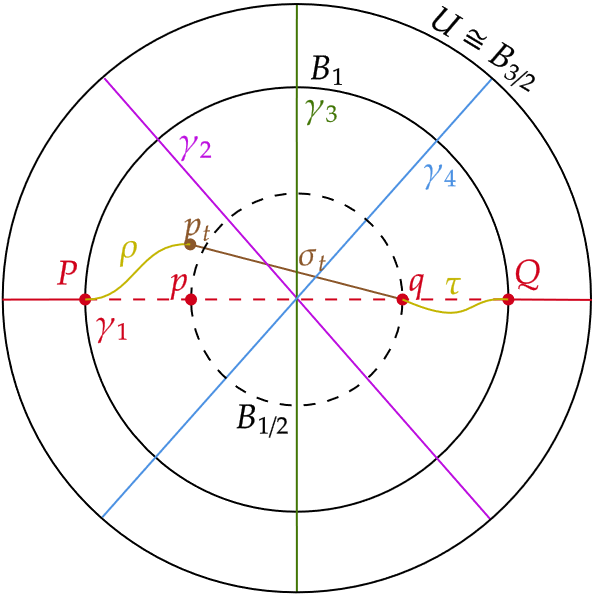}
    \caption{Modified curve with $\rho_i, \sigma_t$ inserted}
    \label{fig:modcurve}
\end{figure}
%

%
%
Now consider a conformal change of metric supported on $B_1(0) \backslash B_{1/2}(0)$ which makes $\gamma_1^*$ a geodesic. Let $k_{\gamma_1^*}$ denote the (scalar) geodesic curvature of $\gamma_1$, and recall that the geodesic curvature transforms as follows under a conformal deformation:
\begin{align*}
g_i & \rightarrow e^{2f} g_i = g_{f,i}\\
\implies k_{\gamma_1^*} & \rightarrow k_{\gamma_1^*, f} = e^{-f}(k_{\gamma_1^*} + \p_n f)
\end{align*}
Note that $k_{\gamma_1^*}$ vanishes at $P$ and $p_{t}$ since we've matched $\gamma_1^*$ with $\gamma_1$ at $P$ and $\sigma_{p_t, q}$ at $p_t$ up to second order. Thus we can choose $f$ which is supported in $B_1(0) \backslash B_{1/2}(0)$. Let $(s,\tilde{t})$ denote Fermi coordinates along $\gamma_1^*$ with $\tilde{t}$ corresponding to the normal coordinate. Let 
\[
d_0 := \min \left[ \min_{i \geq 2} \text{dist}\left( \rho, \gamma_i\right), \min_{i \geq 2} \text{dist} \left( \tau, \gamma_i \right) \right]
\]
By choosing $t$ sufficiently small in equation \eqref{sigmatpath} so that $\rho$ and $\tau$ are sufficiently close to $\gamma_1$, we see that $d_0 > 0$ and is bounded away from $0$ independent of $r_i$ and $t$. We define 
\[
f(s,\tilde{t}):=  - \chi_{d_0/2}(\tilde{t}) \cdot \tilde{t} \cdot  k_{\gamma_1^*}(s) \cdot \psi(s)
\]
Here, $\chi_{d_0/2}(\tilde{t})$ is a positive bump function taking value $1$ for $|\tilde{t}| \leq d_0/2$, which vanishes for $|\tilde{t}| \geq d_0$; $\psi(s)$ is a bump function which is equal to $1$ all along $P \to Q$ and vanishes smoothly outside the neighborhood. Because $k_{\gamma_1^*}(s) = 0$ for all points which are not contained in $B_1(0) \backslash B_{1/2}(0)$, we see that $f$ is supported in $B_1(0) \backslash B_{1/2}(0)$ as well. Thus, $g_{f,i} = g_i$ on $M \backslash U_i$ and $\gamma_1^*$ is a geodesic with respect to $g_{f,i}$. Finally, since 
\[
||u||_{C^{k+2}, g_i} = o(1)
\]
as $t \to 0$, $k_{\gamma_1^*}(s)$ and its higher order derivatives are also $o(1)$, so that $f = o(1) \cdot O(d_0^{-k-2})$ in $C^{k}$ as $t \to 0$. Note that $d_0 = O(r_i)$, but given any choice of $d_0$, we can always choose $t$ smaller so that $||f||_{C^{k}}$ is as small as needed.
Defining $g_f = e^{2f} g$, we have
\[
||g_{f} - g||_{C^{k}} = o(1)
\]
on $U_i$.
Letting $g_{\eps,i} = g_f$, we have constructed a metric satisfying the required properties, and verified the lemma. \end{proof}

If we now consider the corresponding union of geodesic segments $\Gamma^{1,*} = \{\gamma_1^*\} \cup \{\gamma_j\}_{j=2}^d$, then we see that 
$\ord_{\Gamma^{1,*}}(v_i) = d-1$ and $\ord_{\Gamma^{1,*}}(x) \leq 2$ for any $x \in U_i \setminus \{v_i\}$.

To proceed inductively and decrease the order further, we note that in the construction of Lemma \ref{confChangeLem}, $g_{\eps,i} = g$ on $B_{1/2}(0)$, which corresponds to some open neighborhood of $v_{i}$. Thus, to decrease the order further, choose a new open neighborhood $U_{i}^* \ni v_{i}$ such that $U_{i}^* \subseteq \varphi_{i}(B_{1/2}(0))$ \textit{and} such that $U_{i}^* \cap \gamma_1^* = \emptyset$. Now apply Lemma \ref{confChangeLem} again with $U_{i}^*$ to decrease the degree of $v_{i}$ from $d-1$ to $d-2$, etc. (see Figure \ref{fig:degreduc}) 
\begin{figure}[ht!]
    \centering
    \includegraphics[scale=0.3]{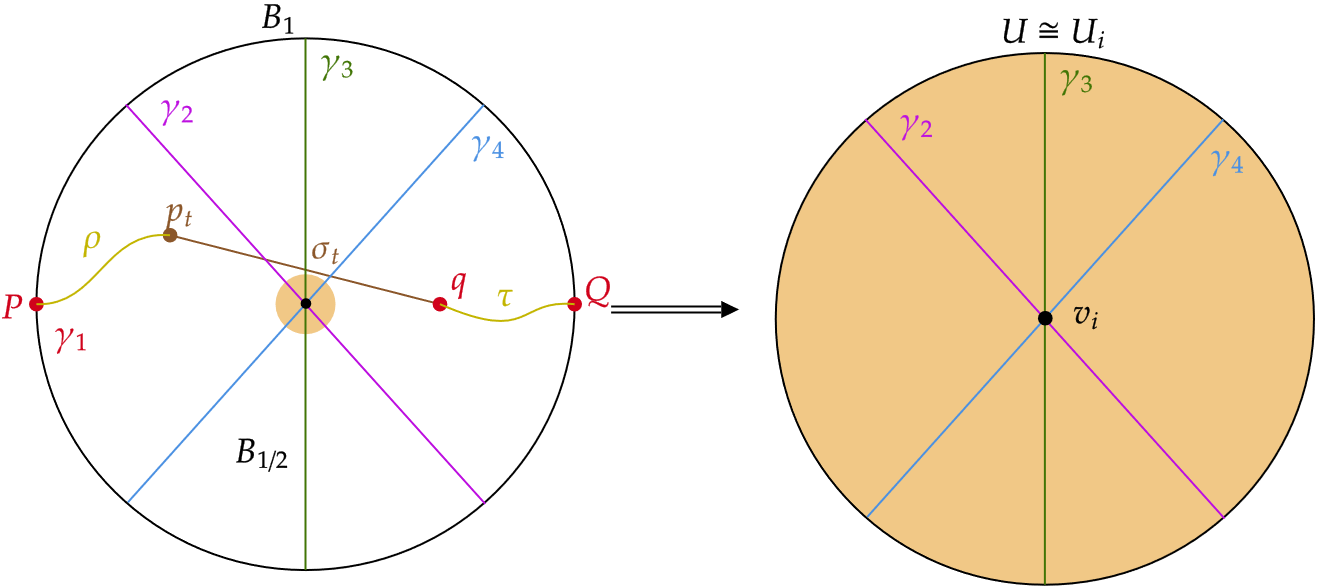}
    \caption{Degree reduction, repeat with new open set}
    \label{fig:degreduc}
\end{figure}

Having handled vertex $v_i$, we note that the change in metric used to lower the order to $2$ (at the cost of increasing the number of vertices by $\binom{d}{2} - 1$) is supported in $U_i$. We now repeat this construction for each $v_j$ on neighborhoods $U_j \ni v_j$ such that $U_i \cap U_j = \emptyset$ and $U_j \cap \{w_{l}\} = \emptyset$. After this finite process, denote the final metric $g_{\eps}$. Therefore, each $\gamma \in S(2L, g)$ has been transformed into a geodesic $\gamma^*$ with respect to $g_\eps$ such that $\mathrm{length}_{g_{\eps}}(\gamma^*) \leq 2L + o(1)$. Moreover, for
\begin{equation*}
\Gamma^* := \{\gamma^* \mid \gamma \in S(2L, g)\},
\end{equation*}
we have
\begin{equation*}
\ord_{\Gamma^*}(x) \leq 2
\end{equation*}
for all $x \in M$.

We now show that having perturbed away higher degree vertices among $S(2L, g)$, no new geodesics appear in $S(L, g_{\eps})$.
\begin{restatable}{lemm}{containLem} \label{containLem}
For $\eps > 0$ sufficiently small and $g_{\eps}$ and $\{\gamma^*\}_{\gamma \in S(2L, g)}$ as above, we have
\[
S(L, g_{\eps}) \subset \{\gamma_*\}_{\gamma \in S(2L, g)}
\]
\end{restatable}
\begin{proof}
Suppose not, then for a sequence of $\eps_k \to 0$, there exists, $\overline{\gamma}_k$, a geodesic with respect to $g_{\eps_k}$ such that 
\[\mathrm{length}_{g_{\eps_k}}(\overline{\gamma}_k) \leq L.\] 
By Lemma \ref{lem:compact_var_met}, there exists a subsequence such that $\overline{\gamma}_k \xrightarrow{C^{2,\alpha}} \gamma_{\infty}$,  a geodesic with respect to $g$ and $\mathrm{length}_g(\gamma_{\infty}) \leq L$. Moreover, $\gamma_{\infty} \in S(2L, g)$ and the convergence of $\{\overline{\gamma}_k\}$ (or their double covers) is graphical by \Cref{lem:compact_var_met}. We also know that for each $\eps_k$, there exists a $\gamma_{k}^* \rightarrow \gamma_{\infty}$ coming from our construction of the conformal perturbation of metric, $g_{\eps_k}$, which is (by construction) a graphical perturbation of $\gamma_{\infty}$. Lemma \ref{lem:Jacobi_fields} gives a contradiction to bumpiness.
\end{proof}

%
%
We have established that $S(L, g_{\eps}) \subset \{\gamma^*\}_{\gamma \in S(2L,g)}$, our perturbed geodesics, which are in $1$-to-$1$ correspondence with $S(2L, g)$.

We finally show that for $\eps > 0$ sufficiently small, any finite cover $\sigma^*$ of $\gamma^*$ with $\mathrm{length}_{g_{\eps}}(\sigma) \leq L$ does not admit a Jacobi field. This claim follows from a similar contradiction argument. The graphical convergence of $\gamma^*_k \to \gamma$ implies graphical convergence (of a subsequence) of $\sigma_k^*$ (a cover of $\gamma_k^*$ with length at most $L$) to $\sigma$ (a cover of $\gamma$ with length at most $L$). Suppose there exists $\varphi_k$, a Jacobi field along $\sigma_k^*$, with $||\varphi_k^*||_{C^0} = 1$ and 
\[
J_{\sigma_k^*, g_{\eps_i}}(\varphi_k) = 0.
\]
By Arzelà-Ascoli, up to a subsequence, $\{\varphi_k\}$ converges to a nonzero $\varphi$ such that
\[
J_{\sigma, g}(\varphi) = 0,
\]
which contradicts bumpiness. This finishes the proof of \Cref{DegFourProp}.
\end{proof}

We now show openness. Let $k \geq 3$ an integer.
%
\begin{restatable}{prop}{openMetProp} \label{openMetProp}
The set of $C^k$ metrics $g$ such that
\begin{itemize}
    \item every $\Gamma \in \mathcal{G}^L_{\mathrm{prim}}(g)$ has vertices with order at most 2,
    \item every $\Gamma \in \mathcal{G}^L(g)$ does not admit a nontrivial Jacobi field.
\end{itemize}
is open in the $C^k$ topology.
\end{restatable}
\begin{proof} Let $g_0$ be such a metric, consider $\{\gamma_i\}_{i=1}^N \subset S(L, g_0)$. Consider the geodesic equation on the space of constant speed, $C^{k,\alpha}$ maps (for some $\alpha > 0$) 
from $S^1 \to M$ and the space of $C^k$-metrics on $M$ 
\begin{align*}
&\Geo: C^{k,\alpha}(S^1, M^2) \times \text{Met}_k(M^2) \to \R \\
&\Geo(\gamma, g) = \nabla^g_{\dot{\gamma}} \dot{\gamma} = \kappa(\gamma, g)
\end{align*}
%
We note that the linearization of $\Geo$ at each $\gamma_i$ is invertible with respect to the $C^{k,\alpha}(S^1, M^2)$ component by the non-degeneracy condition. Thus, the implicit function theorem provides a mapping from $p: g \to \gamma_i(g)$ in an open neighborhood of $(\gamma_i, g_0)$ such that 
\begin{equation} \label{geodesicBound}
||\gamma_i(g_0) - \gamma_i(g)||_{C^{k,\alpha}} \leq K ||g_0 - g||_{C^{k-1,\alpha}}
\end{equation}
and $\gamma_i(g)$ is a geodesic w.r.t. $g$. Here the $C^{k,\alpha}$ bound comes from linearizing the geodesic equation and noting that the Christoffel symbols depends on the metric perturbation in a $C^1$ way. This gives $C^{2,\alpha}$ bounds, and differentiating the geodesic equation $k-2$ times gives the above $C^{k,\alpha}$ bound. \nl 
\indent We now show that (with the non-degeneracy condition) a primitive collection of simple closed geodesics having bounded length and order at most $2$ is an open condition.  
%
Suppose we have a sequence of metrics $g_i \xrightarrow{C^k} g$ and 
\[
\Gamma(g_i) = \{\gamma_{1,i}, \hdots, \gamma_{l_i, i}\} \in \mathcal{G}_{\mathrm{prim}}^L(g_i).
\]
%
By \Cref{lem:compact_var_met} and \Cref{lem:Jacobi_fields}, (up to a subsequence) each $\gamma_{j,i} \to \gamma_j$, a distinct geodesic with respect to $g$ for $j=1,\hdots, l$, with multiplicity exactly one (by \Cref{GeoMultCor}). Let $\Gamma(g) = \{\gamma_1, \hdots, \gamma_l\} \in \mathcal{G}^L_{\mathrm{prim}}(g)$.

For each $v_j \in \Gamma(g)$ with $\ord(v_j) = 2$ (which holds for all vertices of $\Gamma(g)$ by assumption), let $U_j$ be a geodesic ball centered at $v_j$ with radius less than half of the injectivity radius. Therefore, $U_j \cap \{\gamma_i\}$ consists of two geodesics intersecting at a single point. Consider 
\[
S(g) = \left( M \backslash (\sqcup_j U_j) \right) \cap (\cup_i \gamma_i(g_0))
\]
Then $S(g) = \sqcup_{i = 1}^N \overline{\gamma}_i(g)$ consists of $N$ geodesics with boundary, each of which are simple and disjoint from each other. Let $d$ be the minimum distance between all pairs $(\overline{\gamma}_p, \overline{\gamma}_j)$ with $j \neq p$. Then for $g_i$ sufficiently close to $g$, we see that 
\[
S(g_i) = \left( M \backslash (\sqcup_j U_j) \right) \cap (\cup_p \gamma_{i,p}) = \sqcup_{p = 1}^N \overline{\gamma}_{i,p}
\]
will also consist of $N$ geodesics, each not intersecting. Thus, it suffices to show that the order is at most two in each $U_j$. Note that for $||g - g_i||_{C^2}$ sufficiently small, the number of connected components in $U_j \cap \{\gamma_{i,p}\}$ will still be $2$.
By the assumption that $U_j$ is contained in a ball of size smaller than half the injectivity radius, the two connected geodesics in $U_j$ can intersect at most once, and so there exists an $\eps > 0$ such that 
\[
||g - g_i||_{C^2} < \eps \implies \sup_{x \in \Gamma(g_i)} \ord(x) \leq 2,
\]
as desired.
%
\end{proof}

To conclude Theorem \ref{GenSelfIntFull}, we combine Propositions \ref{DegFourProp} and \ref{openMetProp}. Let $\mathcal{M}^L$ denote the set of $C^k$ metrics $g$ for which
\begin{itemize}
    \item every $\Gamma \in \mathcal{G}^L_{\mathrm{prim}}(g)$ has vertices with order at most 2,
    \item every $\Gamma \in \mathcal{G}^L(g)$ does not admit a nontrivial Jacobi field.
\end{itemize}
Then $\mathcal{M}^L$ is open and dense. Since the set of metrics $g$ with $\mathcal{G}_{\mathrm{prim}}(g) = \mathcal{G}_+(g)$ contains
\[
\bigcap_{L \in \N} \mathcal{M}^L,
\]
which is an intersection of open and dense sets, we conclude that the desired set of metrics is $C^k$-generic in the Baire sense. \qed 

\section{Index upper bound} \label{IndexSec}

In this section, we prove Theorem \ref{thm:index_upper}. The strategy of the proof for the $p$-width is as follows. First, we use min-max for the $\varepsilon$-Allen--Cahn energy with the sine-Gordon potential to produce an $\varepsilon$-Allen--Cahn solution having Morse index at most $p$ (see \cite[Theorem 3]{gaspar2018allen}). Second, we take $\eps \to 0$. By \cite[Theorem 3.1]{chodosh2023p}, the associated varifolds converge to a union of closed immersed geodesics (with integer multiplicity). By \cite[Theorem 1.3]{Dey}, the mass of the limit equals the $p$-width. Finally, when the limit union of closed geodesics only has self-intersections of order 2 (which holds generically by Theorem \ref{GenSelfIntFull}\footnote{The general case follows by approximation by generic metrics.}), we apply a new vector field extension argument in the index bound proof of \cite{gaspar2020index,hiesmayr2018index,le2011index,le2015index} to find a subspace of variations of dimension equal to the index of the limit union of closed geodesics over which the second variation of energy of the Allen--Cahn solutions is negative definite for $\eps > 0$ small.

\subsection{Index}

Let $\gamma$ be a primitive closed geodesic. Let $\mathcal{X}^{\perp}(\gamma)$ denote the vector space of normal vector fields along $\gamma$. Let $Q_{\gamma}(\cdot, \cdot)$ be the bilinear form given by the second variation of length of $\gamma$ over $\mathcal{X}^{\perp}(\gamma)$. Let $\mathrm{index}(\gamma)$ be the maximal dimension of a subspace of $\mathcal{X}^{\perp}(\gamma)$ on which $Q_{\gamma}$ is negative definite.

Now let $\Gamma \in \mathcal{G}_{\mathrm{prim}}$. We write
\[ \mathcal{X}^{\perp}(\Gamma) = \bigoplus_{\gamma \in \Gamma} \mathcal{X}^{\perp}(\gamma). \]
For $X \in \mathcal{X}^{\perp}(\Gamma)$ and $\gamma \in \Gamma$, we let $X_{\gamma}$ denote the $\gamma$-summand of $X$. For $X, Y \in \mathcal{X^{\perp}}(\Gamma)$, we define
\[ Q_{\Gamma}(X, Y) = \sum_{\gamma\in\Gamma} Q_{\gamma}(X_{\gamma}, Y_{\gamma}). \]
Then, $\mathrm{index}(\Gamma)$ is defined to be the maximal dimension of a subspace of $\mathcal{X}^{\perp}(\Gamma)$ on which $Q_{\Gamma}$ is negative definite. Equivalently,
\[ \mathrm{index}(\Gamma) = \sum_{\gamma \in \Gamma} \mathrm{index}(\gamma). \]

If $\tilde{\Gamma} \in \mathcal{G}$ is not primitive, then we let $\mathrm{index}(\tilde{\Gamma}) = \mathrm{index}(\Gamma)$ for any primitive representative $\Gamma \in \mathcal{G}_{\mathrm{prim}}$ of $\tilde{\Gamma}$ (see Remark \ref{rem:prim}). Since the index of a primitive closed geodesic is parametrization-independent, this notion of index is well-defined.

We record the fact that index is lower semicontinuous under smooth convergence. 

\begin{lemma}\label{lem:index_lower_semi}
    Suppose $\{g_i\}_{i\in\N}$ is a sequence of metrics on $M$ converging smoothly to $g$. Suppose $\{\Gamma_i\}_{i \in \N} \subset \mathcal{G}_{\mathrm{prim}}(g_i)$ converges smoothly to $\Gamma \in \mathcal{G}(g)$. Then
    \[ \mathrm{index}_g(\Gamma) \leq \liminf_{i \to \infty} \mathrm{index}_{g_i}(\Gamma_i). \]
\end{lemma}
\begin{proof}
    Suppose $\{\gamma_i\}_{i \in \N}$ is a sequence of closed immersed geodesics in $(M, g_i)$ converging smoothly in the graphical sense to a closed immersed geodesic $\gamma$ in $(M, g)$.

    There are $\eps,\ \delta > 0$ so that for every $\theta \in S^1$, $\gamma\mid_{(\theta - \eps, \theta + \eps)}$ is an embedding in $B_\delta(\gamma(\theta))$. By compactness, we can find finitely many $\theta_1, \hdots, \theta_l \in S^1$ so that $\{(\theta_j - \eps/2, \theta_j + \eps/2)\}_{j=1}^l$ covers $S^1$ and $\{B_\delta(\gamma(\theta_j))\}_{j=1}^l$ covers $B_{\delta/2}(\gamma(S^1))$.
    
    Let $X \in \mathcal{X}^{\perp}(\gamma)$. For each $j = 1,\hdots, l$, we let $\tilde{X}_j$ be the extension of $X$ to the $\delta/4$ normal tubular neighborhood of $\gamma\mid_{(\theta_j - \eps, \theta_j + \eps)}$ given by parallel transport along normal geodesics. Since $\tilde{X}_j$ is constructed canonically in each neighborhood, $\tilde{X}_j$ agrees with $\tilde{X}_{j+1}$ and $\tilde{X}_{j-1}$ on the overlaps.
    For $i$ sufficiently large (so that $\gamma_i(S^1)$ lies in $B_{\delta/8}(\gamma(S^1))$), we define $X_i \in \mathcal{X}^{\perp}(\gamma_i)$ by taking the perpendicular component of the vector field along $\gamma_i$ given by $\tilde{X}_j(\gamma_i(\theta))$ if $j$ minimizes $|\theta - \theta_j|$. By smooth convergence and Proposition \ref{prop:variation_formulas},
    \[ \lim_{i \to \infty} Q_{\gamma_i}^{g_i}(X_i, X_i) = Q_{\gamma}^g(X, X). \]
    
    Applying this extension construction to a basis for a subspace of $\mathcal{X}^{\perp}(\gamma)$ of maximal dimension on which $Q_{\gamma}^g$ is negative definite implies
    \[ \mathrm{index}_g(\gamma) \leq \liminf_{i \to \infty} \mathrm{index}_{g_i}(\gamma_i). \]

    To conclude the lemma, we need only confirm that if $\tilde{\gamma}$ is not primitive and $\gamma$ is a primitive geodesic with the same image as $\tilde{\gamma}$, then $\mathrm{index}_g(\gamma) \leq \mathrm{index}_g(\tilde{\gamma})$. Indeed, a basis of a subspace of $\mathcal{X}^{\perp}(\gamma)$ on which $Q_{\gamma}^g$ is negative definite can be extended to a basis of a subspace of $\mathcal{X}^{\perp}(\tilde{\gamma})$ of the same dimension on which $Q_{\tilde{\gamma}}^g$ is negative definite. Namely, given $X \in \mathcal{X}^{\perp}(\gamma)$, we take $\tilde{X}(\tilde{\theta}) = X(\theta)$ for any $\tilde{\theta},\ \theta \in S^1$ satisfying $\tilde{\gamma}(\tilde{\theta}) = \gamma(\theta)$. By Proposition \ref{prop:variation_formulas}, $Q_{\tilde{\gamma}}^g(\tilde{X}, \tilde{X}) = mQ_{\gamma}^g(X,X)$, where $m$ is the number of times $\tilde{\gamma}$ traverses the image of $\gamma$.
\end{proof}

We also record lower semi-continuity of the weighted vertex count.

\begin{lemma}\label{lem:weighted_vertex_lower_semi}
    Suppose $\{g_i\}_{i\in\N}$ is a sequence of metrics on $M$ converging smoothly to $g$. Suppose $\{\Gamma_i\}_{i \in \N} \subset \mathcal{G}_{\mathrm{prim}}(g_i)$ converges smoothly to $\tilde{\Gamma} \in \mathcal{G}(g)$ with primitive representative $\Gamma \in \mathcal{G}_{\mathrm{prim}}(g)$. Then
    \[ \sum_{x\in\mathrm{Vert}(\Gamma)} \binom{\mathrm{ord}_{\Gamma}(x)}{2} \leq \liminf_{i \to \infty} \sum_{x\in\mathrm{Vert}(\Gamma_i)} \binom{\mathrm{ord}_{\Gamma_i}(x)}{2}. \]
\end{lemma}
\begin{proof}
    Choose $\eps$ small enough so that $B_\eps(x)$ is disjoint from $B_\eps(y)$ for all $x \neq y \in \mathrm{Vert}(\Gamma)$ and $\Gamma$ in $B_{\eps}(x)$ consists of $\mathrm{ord}_{\Gamma}(x)$ geodesic segments intersecting at $x$ for all $x \in \mathrm{Vert}(\Gamma)$.

    Fix $x \in \mathrm{Vert}(\Gamma)$, and let $\{\sigma_1, \hdots, \sigma_{\mathrm{ord}_{\Gamma}(x)}\}$ be the geodesic segments making up $\Gamma$ in $B_{\eps}(x)$. By smooth convergence, for all $i$ sufficiently large, there are components $\{\sigma_1^i, \hdots, \sigma_{\mathrm{ord}_{\Gamma}(x)}^i\}$ of $\Gamma_i$ in $B_\eps(x)$ so that $\sigma_j^i$ converges smoothly and graphically to $\sigma_j$. By smooth convergence, if $j_1 \neq j_2$, then $\sigma_{j_1}^i$ intersects $\sigma_{j_2}^i$ at some $x^i_{j_1, j_2} \in B_\eps(x)$ for all $i$ sufficiently large. Hence,
    \[ \binom{\mathrm{ord}_{\Gamma}(x)}{2} \leq \liminf_{i \to \infty} \sum_{y \in \mathrm{Vert}(\Gamma_i) \cap B_{\eps}(x)} \binom{\mathrm{ord}_{\Gamma_i}(y)}{2}. \]
    The desired inequality follows by summing over the vertices of $\Gamma$.
\end{proof}

\subsection{Vector field extension}

Given $\Gamma \in \mathcal{G}_+$ and $X \in \mathcal{X}^{\perp}(\Gamma)$, we aim to find a smooth ambient vector field $\tilde{X} \in \mathcal{X}(M)$ with the property that the normal component of the restriction of $\tilde{X}$ to any $\gamma \in \Gamma$ agrees with $X_{\gamma}$. Since the second variation of length only depends on the normal component of the variation vector field (see Proposition \ref{prop:variation_formulas}), the second variation of the length of $\Gamma$ where each $\gamma \in \Gamma$ varies by $X_{\gamma}$ agrees with the second variation of the length of $\Gamma$ along the ambient flow generated by $\tilde{X}$.

For the index estimate of the next subsection, we also require that this vector field extension has good estimates on the tangential component of the restriction of $\tilde{X}$ to $\gamma \in \Gamma$ and the normal derivatives of $\tilde{X}$ along $\gamma \in \Gamma$.

\begin{restatable}{lemm}{lem:extension}\label{lem:extension}
    Let $\Gamma \in \mathcal{G}_+$ and $X \in \mathcal{X}^{\perp}(\Gamma)$. For any $\delta > 0$, there is a smooth vector field $\tilde{X}$ on $M$ so that
    \begin{equation}\label{eqn:extension}
        (\tilde{X}\mid_{\gamma})^{\perp} = X_\gamma \ \ \text{for all}\ \ \gamma \in \Gamma,
    \end{equation}
    \begin{equation}\label{eqn:same_2nd} Q_{\Gamma}(X, X) = \delta^2\Gamma(\tilde{X}, \tilde{X}) \end{equation}
    and
    \begin{align}\label{eqn:2nd_error}
        & \sum_{\gamma} \||\nabla_{n_\gamma}\tilde{X}| + |(\tilde{X}\mid_{\gamma})^T|\|_{L^{\infty}(\gamma)} \leq C(M, \Gamma, X),\\
        & \sum_{\gamma} \cH^1(\mathrm{spt}(|(\nabla_{n_\gamma}\tilde{X})\mid_\gamma| + |(\tilde{X}\mid_{\gamma})^T|)) \leq \delta,\notag
    \end{align}
    %
    where $n_{\gamma}$ is any measurable choice of unit normal vector field along $\gamma \in \Gamma$ and $\mathcal{H}^1$ is the 1-dimensional Hausdorff measure.
\end{restatable}
\begin{proof}
    Since the second variation of length only depends on the normal component of the vector field (see Proposition \ref{prop:variation_formulas}), \eqref{eqn:extension} implies \eqref{eqn:same_2nd}. Hence, it suffices to find a smooth vector field $\tilde{X}$ satisfying \eqref{eqn:extension} and \eqref{eqn:2nd_error}.

    Let $\{x_1, \hdots, x_m\}$ be the self-intersections of $\Gamma$. For $1 \leq j \leq m$, let $\gamma_{0,j},\ \gamma_{1,j} \in \Gamma$ and $s_{0,j},\ s_{1,j} \in S^1$ be the unique choices so that $\gamma_{0,j}(s_{0,j}) = \gamma_{1,j}(s_{1,j}) = x_j$ and either $\gamma_{0,j} \neq \gamma_{1,j}$ or $s_{0,j} \neq s_{1,j}$ (we can do this by the assumption $\Gamma \in \mathcal{G}_{+}$).

    Choose $r > 0$ sufficiently small so that
    \begin{itemize}
        \item $\{B_r(x_j)\}_j$ are pairwise disjoint,
        \item the image of $\Gamma$ in $B_r(x_j)$ consists of two connected geodesic segments intersecting at $x_j$ for all $1 \leq j \leq m$,
        \item for each $1 \leq j \leq m$, there is a diffeomorphism
        \[ \phi_j : B_r(x_j) \to B_r(0) \subset \R^2\]
        so that
        \begin{itemize}
            \item $\phi_j(x_j) = 0$,
            \item $\phi_j(\bigcup_{\gamma \in \Gamma} \gamma(S_1)) = \{xy=0\} \cap B_r(0)\subset \R^2$,
            \item there is a constant $c = c(M, \Gamma) < \infty$ so that $c^{-1} \leq |\nabla \phi_j| \leq c$.
        \end{itemize}
    \end{itemize}
    The existence of $r$ and $\phi_j$ follows from the fact that the exponential map at $x_j$ maps the two intersecting geodesic segments at $x_j$ to two straight lines through the origin, so we can then apply a linear transformation of $\R^2$ to map these lines to the coordinate axes.

    Now take any $\eta \in (0, r)$. We construct a smooth vector field $\tilde{Y}_j$ on $B_\eta(x_j)$ satisfying
    \begin{itemize}
        \item $(\tilde{Y}_j\mid_{\gamma_{k,j}})^{\perp} = X_{\gamma_{k,j}}$ for $k \in \{0, 1\}$,
        \item $|\tilde{Y}_j| + |\nabla \tilde{Y}_j| \leq C$ for some constant $C = C(M, \Gamma, X) < \infty$ (independent of $\eta$).
    \end{itemize}
    Consider in $T_{x_j}M$ the lines
    \[ L_{0,j} = \{X_{\gamma_{0,j}}(s_{0,j}) + t\gamma_{0,j}'(s_{0,j}) \mid t \in \R\},\ \ L_{1,j} = \{X_{\gamma_{1,j}}(s_{1,j}) + t\gamma_{1,j}'(s_{1,j}) \mid t \in \R\}. \]
    Since the self-intersections are transverse, these lines have a unique intersection
    \[ X_{\gamma_{0,j}}(s_{0,j}) + t_{0,j}\gamma_{0,j}'(s_{0,j}) = X_{\gamma_{1,j}}(s_{1,j}) + t_{1,j}\gamma_{1,j}'(s_{1,j}). \]
    see Figure \ref{fig:normExt}.
    \begin{figure}[ht!]
        \centering
        \includegraphics[scale=0.45]{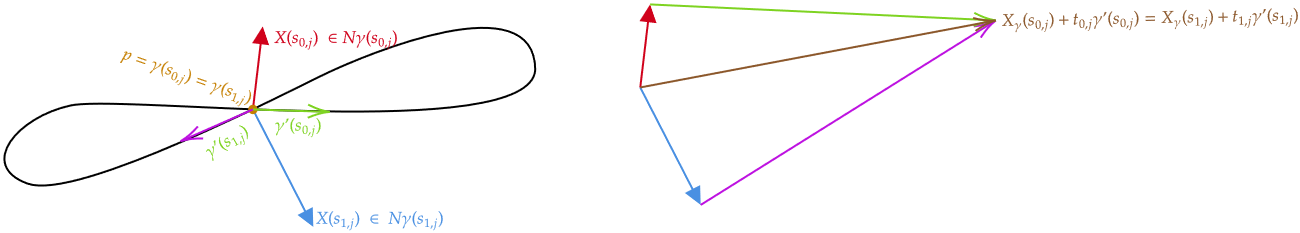}
        \caption{Brown arrow shows the correct choice for the extension of the normal vector field to a vector field on $M$.}
        \label{fig:normExt}
    \end{figure}
    Let $T_{k,j}$ be the tangential vector field along $\gamma_{k,j}$ in $B_\eta(x_j)$ given by parallel transport of $t_{k,j}\gamma_{k,j}'(s_{k,j}) \in T_{x_j}M$ for $k \in \{0, 1\}$. Let $Y_{k,j} = X_{\gamma_{k,j}} + T_{k,j}$ for $k \in \{0,1\}$, which are smooth vector fields satisfying
    \[ Y_{0,j}(s_{0,j}) = Y_{1,j}(s_{1,j}). \]
    Applying the diffeomorphism $\phi_j$, using Proposition \ref{prop:extension}, and then applying $\phi_j^{-1}$, we obtain the desired extension $\tilde{Y}_j$ on $B_\eta(x_j)$.

    Using normal parallel transport of $X$ and smooth cutoffs, we can construct a smooth vector field $\tilde{Z}$ on $M \setminus \bigcup_j B_{\eta/2}(x_j)$ satisyfing
    \begin{itemize}
        \item $(\tilde{Z}\mid_{\gamma})^{\perp} = X_\gamma$ for all $\gamma \in \Gamma$,
        \item $(\tilde{Z}\mid_{\gamma})^T = 0$ for all $\gamma \in \Gamma$,
        \item $\nabla_{n_\gamma}\tilde{Z} = 0$ along $\gamma$ for all $\gamma \in \Gamma$.
    \end{itemize}

    Let $\psi_j : B_\eta(x_j) \to [0, 1]$ be a smooth radial function with support in $B_{7\eta/8}(x_j)$ and $\psi_j\mid_{B_{5\eta/8}(x_j)} \equiv 1$. We define
    \[ \tilde{X} = \sum_{j=1}^m \psi_j\tilde{Y}_j + (1-\psi_j)\tilde{Z}. \]
    Since $\nabla_{n_\gamma}\psi_j \equiv 0$ along $\gamma$ for all $\gamma \in \Gamma$ (by choosing $\psi_j$ radial, we conclude that $\tilde{X}$ satisfies
    \begin{itemize}
        \item $(\tilde{X}\mid_{\gamma})^\perp = X_\gamma$ for all $\gamma \in \Gamma$ (i.e.\ \eqref{eqn:extension} holds),
        \item $|(\tilde{X}\mid_{\gamma})^T| + |\nabla_{n_\gamma}\tilde{X}| \leq C(M, \Gamma, X) < \infty$ (independent of $\eta$) along $\gamma$ for all $\gamma \in \Gamma$,
        \item $|(\tilde{X}\mid_{\gamma})^T| + |\nabla_{n_\gamma}\tilde{X}| \equiv 0$ outside $B_\eta(x_j)$ along $\gamma$ for all $\gamma \in \Gamma$.
    \end{itemize}
    We deduce \eqref{eqn:2nd_error} by taking $\eta$ sufficiently small.
\end{proof}

\begin{restatable}{prop}{prop:extension} \label{prop:extension}
    Suppose we have a map $u : \{(x,y)\in\R^2 \mid xy=0\} \to \R^2$ so that $u(t, 0)$ and $u(0, t)$ are both smooth in $t$. Then there is a smooth map $U : \R^2 \to \R^2$ so that $U\mid_{\{xy=0\}} = u$. Moreover, we have
    \[ |(\nabla U)(x,y)| \leq \max\{|(\partial_xu)(x,0)|, |(\partial_yu)(0,y)|, |(\partial_xu)(0,0)|, |(\partial_yu)(0,0)|\}. \]
\end{restatable}
\begin{proof}
    We define
    \[ U(x, y) = u(x,0) + u(0,y) - u(0,0). \]
    We see
    \[ U(x, 0) = u(x,0) + u(0,0) - u(0,0) = u(x,0) \]
    and
    \[ U(0, y) = u(0,0) + u(0, y) - u(0,0) = u(0,y). \]
    The smoothness of $U$ follows immediately from the smoothness of $u$. The gradient bound follows directly from the formula.
\end{proof}

\subsection{Index bound}

Consider a sequence $\{u_k\}$ of solutions to the Allen--Cahn equation with parameter $\eps_k \to 0$ satisfying
\[ \limsup_k \sup_M |u_k| \leq c_0,\ \ \limsup_k E_{\eps_k}(u_{\eps_k}) \leq E_0,\ \ \limsup_k \mathrm{index}_{E_{\eps_k}}(u_k) \leq p. \]
Suppose further that the associated varifold $V_{u_k}$ converge to a varifold $V$ given by
\[ V = \sum_{\gamma \in \Gamma} m_{\gamma}|\gamma| \]
for some $\Gamma \in \mathcal{G}_+$.

Using the vector field extension from Lemma \ref{lem:extension} in the index bound proof of \cite{gaspar2020index} (see also \cite{hiesmayr2018index,le2011index,le2015index}), we show that $\mathrm{index}(\Gamma) \leq p$.

\begin{theorem}\label{thm:index_general}
    Let $\Gamma \in \mathcal{G}_+$. Suppose
    \[ V = \sum_{\gamma \in \Gamma} m_{\gamma}|\gamma| \]
    for some $m_{\gamma} \in \N$ is the limit of the associated varifolds of a sequence $\{u_k\}$ of solutions to the Allen--Cahn equation with parameter $\eps_k \to 0$ satisfying
    \[ \limsup_k \sup_M |u_k| \leq c_0,\ \ \limsup_k E_{\eps_k}(u_{\eps_k}) \leq E_0,\ \ \limsup_k \mathrm{index}_{E_{\eps_k}}(u_k) \leq p. \]
    Then $\mathrm{index}(\Gamma) \leq p$ and $\#\mathrm{Vert}(\Gamma) \leq p$.
\end{theorem}
\begin{proof}
    \emph{Index bound}. Let $l = \mathrm{index}(\Gamma)$. Let $\{X^1, \hdots, X^l\} \subset \mathcal{X}^{\perp}(\Gamma)$ be a basis for a subspace of $\mathcal{X}^{\perp}(\Gamma)$ on which $Q_{\Gamma}$ is negative definite.

    Define $\bar{X}^i \in \mathcal{X}^{\perp}(\Gamma)$ by
    \[ \bar{X}^i_{\gamma} = m_{\gamma}^{-1/2}X^i_{\gamma}. \]
    By construction, $\{\bar{X}^1, \hdots, \bar{X}^l\}$ is a basis for a subspace of $\mathcal{X}^{\perp}(\Gamma)$ on which
    \[ Q_V(X, Y) := \sum_{\gamma \in \Gamma} m_{\gamma} Q_\gamma(X_{\gamma}, Y_{\gamma}) \]
    is negative definite. Let $\eta > 0$ satisfy
    \[ \max_{a \in S^{l-1} \subset \R^l} Q_V(a_1\bar{X}^1+\hdots+a_l\bar{X}^l, a_1\bar{X}^1+\hdots+a_l\bar{X}^l) \leq -\eta < 0. \]

    Take $\eps > 0$. Let $\tilde{X}^i$ be the smooth vector field constructed from $\bar{X}^i$ in Lemma \ref{lem:extension} with error $\eps$. By construction, $\{\tilde{X}^1, \hdots, \tilde{X}^l\}$ is a basis for a subspace of the space of smooth vector fields on $M$ on which $\delta^2V$ is negative definite. In particular, by Proposition \ref{prop:variation_formulas}, we have
    \[ \max_{a \in S^{l-1} \subset \R^l} \delta^2V(a_1\tilde{X}^1+\hdots+a_l\tilde{X}^l) \leq -\eta < 0. \]
    
    Let $a \in S^{l-1} \subset \R^l$ and 
    \[ Y = a_1\tilde{X}^1 + \hdots + a_l\tilde{X}^l. \]
    By \cite[Proposition 3.3]{gaspar2020index} and Lemma \ref{lem:extension}, we have
    \begin{align*}
        \frac{1}{2\sigma}\lim_{k \to \infty} \delta^2E_{\eps_k}(u_{\eps_k}, Y)
        & = \delta^2V(Y) + \sum_{\gamma \in \Gamma} m_{\gamma} \int_{\gamma} (\langle \nabla_{n_\gamma} Y, n_\gamma\rangle^2 + R(Y, n_\gamma, n_\gamma, Y))\ d\cH^1\\
        & \leq -\eta + C(M, \Gamma, V)\delta^2.
    \end{align*}
    Taking $\delta^2 < \eta/C(M,\Gamma,V)$, we see that $\delta^2E_{\eps_k}(u_{\eps_k}, \cdot)$ is negative definite on the subspace of smooth vector fields spanned by $\{\tilde{X}^1, \hdots, \tilde{X}^l\}$ for all $k$ sufficiently large. Hence, $l \leq p$, as desired.

    \emph{Vertex bound}. The vertex bound is an immediate consequence of \cite{tonegawa2005stable,TonegawaWickramasekera,guaraco2018min}. Indeed, suppose for contradiction that $\#\mathrm{Vert}(\Gamma) = v > p$. Let $r > 0$ sufficiently small so that $\{B_r(x_i)\}_{i=1}^{v}$ are pairwise disjoint, where $\{x_i\}_{i=1}^{v} = \mathrm{Vert}(\Gamma)$ are the vertices of $\Gamma$. Since $\mathrm{index}_{E_{\eps_k}}(u_k) \leq p$ for all $k$, there is a subsequence (not relabeled) and some $i^* \in \{1, \hdots, v\}$ so that $u_k$ is $E_{\eps_k}$-stable in $B_r(x_{i^*})$ for all $k$. Then by \cite[Theorem 5]{tonegawa2005stable} and \cite[Theorem 2.1]{TonegawaWickramasekera} (with the appropriate modification to ambient Riemannian surfaces as in \cite[Appendix B]{guaraco2018min}), the limit is smoothly embedded in $B_r(x_{i^*})$ (i.e.\ no self-intersections), which yields a contradiction. 
\end{proof}

Equipped with Theorem \ref{thm:index_general}, we can now prove Theorem \ref{thm:index_upper}.

\begin{proof}[Proof of Theorem \ref{thm:index_upper}]
    Since the $p$-widths are continuous in the metric (see \cite[Lemma 2.1]{irie2018density}), the compactness result Lemma \ref{lem:compact_var_met}, the generic metric result Theorem \ref{GenSelfIntFull}, and the lower semicontinuity of index and weighted vertex results Lemmas \ref{lem:index_lower_semi} and \ref{lem:weighted_vertex_lower_semi} imply that it suffices to prove the theorem for metrics satisfying $\mathcal{G}_{\mathrm{prim}} = \mathcal{G}_+$. 
%

    Let $\{\Pi_i\}_i$ be a sequence of Allen--Cahn homotopy classes from $p$-dimensional cubical complexes satisfying $\lim_{i \to \infty} \mathbf{L}_{\mathrm{PT}}(\Pi_i) = \omega_p(M, g)$. Again by the compactness result Lemma \ref{lem:compact_var_met} and the lower semicontinuity of index and weighted vertex results Lemmas \ref{lem:index_lower_semi} and \ref{lem:weighted_vertex_lower_semi}, it suffices to prove the existence of $\Gamma_p^i \in \mathcal{G}_{\mathrm{prim}}$ and positive integers $\{m_{\gamma}\}_{\gamma \in \Gamma_p^i}$ satisfying
    \[ \mathbf{L}_{\mathrm{PT}}(\Pi_i) = \sum_{\gamma \in \Gamma_p^i} m_{\gamma} \mathrm{length}(\gamma), \]
    and
    \[ \mathrm{index}(\Gamma_p^i) \leq p\ \ \text{and}\ \ \sum_{x \in \mathrm{Vert}(\Gamma_p^i)} \binom{\mathrm{ord}_{\Gamma_p^i}(x)}{2} \leq p. \]
    Since we assume $\mathcal{G}_{\mathrm{prim}} = \mathcal{G}_+$, the weighted vertex bound is equivalent to $\#\mathrm{Vert}(\Gamma_p^i) \leq p$.
    
    For each $i$, by \cite[Theorem 1.2]{chodosh2023p}, min-max for the Allen--Cahn energy with sine-Gordon potential produces a sequence $\{u_k^i\}_k$ of solutions to the Allen--Cahn equation with parameter $\eps_k^i \to 0$ satisfying
    \[ \limsup_k \sup_M |u_k^i| \leq c_0,\ \ \limsup_k E_{\eps_k}(u_{\eps_k^i}) \leq E_0,\ \ \limsup_k \mathrm{index}_{E_{\eps_k^i}}(u_k^i) \leq p \]
    and $V_{u_k^i} \rightharpoonup V^i$ for a varifold
    \[ V^i = \sum_{\gamma \in \Gamma_p^i} m_{\gamma} |\gamma| \]
    for some $\Gamma_p^i \in \mathcal{G}_{\mathrm{prim}}$ satisfying
    \[ \sum_{\gamma \in \Gamma_p^i} m_{\gamma} \mathrm{length}(\gamma) = \mathbf{L}_{\mathrm{PT}}(\Pi_i). \]
    By the assumption that $\mathcal{G}_{\mathrm{prim}} = \mathcal{G}_+$, we have $\Gamma_p^i \in \mathcal{G}_+$. Then Theorem \ref{thm:index_general} implies 
    \[
    \mathrm{index}(\Gamma_p^i) \leq p\ \ \text{and}\ \ \#\mathrm{Vert}(\Gamma_p^i) \leq p,
    \]
    as desired. 
\end{proof}
\section{Higher multiplicity} \label{HighMultSec}
In this section, we construct a sequence $\{g_p\}$ of metrics on $S^2$ with the following properties: 
\begin{itemize}
\item for any $p\in\N$, $(S^2,g_p)$ has strictly positive Gauss curvature everywhere, and
\item there is an open neighbourhood $U_p$ of $g_p$ in the space of smooth Riemannian metrics on $S^2$ such that for all $g\in U_p$ the first $p$ widths of the length functional satisfy
\[\omega_l(S^2, g)=l\,\omega_1(S^2, g)\text { for } l=1, \dots, p,\]
and $\omega_l(S^2, g)$ can only be achieved by a closed geodesic $\gamma^g_0$ with multiplicity $l$, for all $l=1, \dots, p$.
\end{itemize}
First we recall a sweepout construction by Guth \cite[Example 2]{guth2009minimax}
\subsection{Guth's sweepout construction} \label{GuthEx}
\begin{figure}[ht!]
\centering
\includegraphics[scale=0.4]{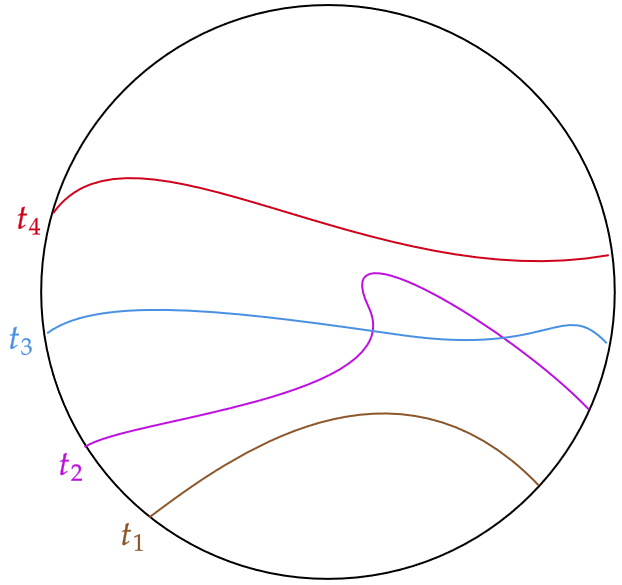}
\caption{Guth's $p$-sweepout construction}
\label{fig:GuthEx}
\end{figure}
Let $\sigma: [0,1] \to \mathcal{Z}_1(M^2, \Z_2)$ be a $1$-sweepout of $M^2$. Then we can construct a $p$-sweepout naturally (see figure \ref{fig:GuthEx}). Define 
\begin{align*}
D_p &= \{(t_1, t_2, \dots, t_p) \; | \; 0 < t_1 < t_2 < \dots < t_p < 1 \}\subset [0,1]^p \\
\sigma_p: & D_p \to \mathcal{Z}_1(M^2, \Z_2) \\
\sigma_p(t_1, \dots, t_p) &= \sum_{i = 1}^p \sigma(t_i)
\end{align*}
%
%
As in \cite[Example 2]{guth2009minimax}, $\sigma_p$ extends to a map on $\overline{\sigma}_p: \overline{D}_p \to \mathcal{Z}_1(M^2, \Z_2)$ and this satisfies the cohomology condition, $\overline{\sigma}_p^*(\lambda^p) \neq 0$ where $\lambda$ is a nontrivial generator of $H^1(\mathcal{Z}_1(M^2, \Z_2))$ (see \S 5 of \cite{guth2009minimax} for details). Thus $\overline{\sigma}_p$ is a $p$-sweepout. Moreover
\[
\sup_{t \in D_p} ||\sigma_p(t)|| \leq p \sup_{t \in [0,1]} ||\sigma(t)||
\]
In particular this demonstrates that 
\begin{equation} \label{pWidthUpper}
\omega_p \leq p \omega_1
\end{equation}
by choosing $\sigma(t)$ such that $\sup_{t \in [0,1]} ||\sigma(t)|| \leq \omega_1 + \eps$ and sending $\eps \to 0$. We note that a similar construction from Dey's proof of Theorem 1 in \cite{dey2023existence} also constructs $p$-sweepouts from $1$-sweepouts and proves \eqref{pWidthUpper}. In essence, one can inductively apply the suspension construction of \cite[\S 4]{dey2023existence} to a 1-sweepout close to $\omega_1$.

\subsection{Construction of elongated metrics}
The construction of the metrics $g_l$ is modelled on the construction in \cite{WangZhou_mult2}, which yields non-bumpy metrics on $S^{n+1}$ ($3\leq n+1\leq 7$) for which the second width of area can only be achieved by an embedded minimal $S^n$ with multiiplicity 2.

Fix $\mu\in[1,+\infty)$, and let $l\in\Z_{>0}$. Consider the sequence $\{M_k\}$ of surfaces in $\R^3$ given by
\begin{equation}\label{eq:M_k}
    M_k=\left\{(x_1,x_2,x_3)\in\R^3\mid x_1^2+x_2^2+\frac{x^{2\mu}_3}{k}=1\right\},
\end{equation} 
each endowed with the metric induced by the embedding into $\R^3$. Let $\gamma_0:S^1\to M_k$ be the parametrized simple curve
\begin{equation}\label{eq:gamma_0}
\gamma_0(\theta)=(\cos\theta, \sin\theta, 0),
\end{equation}
which has image $\gamma_0(S^1)=M_k\cap\{x_3=0\}$.
Note that for each $k\geq 1$, $M_k$ is diffeomorphic to $S^2$, and $\gamma_0$ is a simple closed geodesic in $M_k$. 

The main result of this subsection is to prove \Cref{HigherMultCounter}, which we state in more detail as the following:
\begin{restatable}{thmm}{thm:higher_mult_example} \label{thm:higher_mult_example}
    Fix $\mu\in[1,+\infty)$, and let $p\in\N$. There is $k_p$ depending on $p$ and $\mu$ only such that for all $k\geq k_p$ the following holds: for all $l=1,\dots, p$
    \[\omega_l(M_k)=l\,\mathrm{length}(\gamma_0),\]
    and $\omega_l(M_k)$ can only be achieved by $l|\gamma_0|$, i.e. the stationary 1-varifold induced by $\gamma_0(S^1)$ with multiplicity $l$.
\end{restatable}

First, let us list some properties of the surfaces $M_k$ which will be useful in the proof of Theorem \ref{thm:higher_mult_example}.
\begin{restatable}{lemm}{lem:M_k_prop} \label{lem:M_k_prop}
    Let $k\geq 1$, and let $M_k$ be defined as in \eqref{eq:M_k}, endowed with the metric induced by the embedding into $\R^3$. Then, the following properties hold:
    \begin{enumerate}
        \item If $\mu=1$, $K_{M_k}>0$. Moreover, $K_{M_k}=k^{-1}>0$ on $\gamma_0(S^1)=M_k\cap\{x_3=0\}$.  \\
        If $\mu\in(1,+\infty)$, then $K_{M_k}\geq 0$, and, for $x\in M_k$ 
        \begin{itemize}
            \item $K_{M_k}(x)=0$ if and only if $x\in M_k\cap\{x_3=0\}$,
            \item $K_{M_k}(x)>0$ otherwise.
        \end{itemize}
        \item $M_k$ converges locally smoothly to the cylinder $S^1(1)\times\R\subset\R^2\times \R=\R^3$ as $k\to\infty$. Therefore,  for any $p\in\Z_{>0}$,
        \begin{equation}\label{eq:width_convergence}
            \omega_p(M_k)\rightarrow\omega_p(S^1(1)\times\R)=2\pi p
        \end{equation}
        as $k\to\infty$.
        \item $\gamma_0$ is 
        \begin{itemize}
            \item a simple closed geodesic of index 1 in $M_k$, if $\mu=1$;
            \item a degenerate stable simple closed geodesic in $M_k$, $\mu\in(1,+\infty)$.
        \end{itemize}
        \item Each connected component of $M_k\setminus\{x_3=0\}$ is foliated by the simple closed curves $M_k\cap\{|x_3|=c\}$ for $0<c<k^{1/{2\mu}}$. For each $0<c<k^{1/{2\mu}}$, the simple closed curves $M_k\cap\{|x_3|=c\}$ have everywhere nonzero curvature vector, pointing towards $\{x_3=0\}$. Therefore, for each $0<c<k^{1/{2\mu}}$, there can be no closed geodesics entirely contained in $\{x_3\geq c\}$ or $\{x_3\leq -c\}$.
        \item Every closed geodesic in $M_k$ intersects $\gamma_0(S^1)$.
    \end{enumerate}
\end{restatable}

\begin{proof}
    Property (1) and the fact that $\gamma_0$ is a simple closed geodesic are clear from the definition of $M_k$. Property (2) is well known but see (\cite[\S2]{WangZhou_mult2} or a similar argument in \cite[Lemma 6]{song2018existence}).
    Property (3) follows directly from the second variation formula in Proposition \ref{prop:variation_formulas}. Property (4) follows from direct calculation and the maximum principle.
    Finally, property (5) follows from (4).
\end{proof}

We shall now prove Theorem \ref{thm:higher_mult_example}.
\begin{proof}[Proof of Theorem \ref{thm:higher_mult_example}]
First of all, note that for every $k\geq 1$, we have $\omega_1(M_k)=\length(\gamma_0)=2\pi$. Applying Guth's $p$-sweepout construction from \S \ref{GuthEx}, we have that for every $p\geq 1$
\[\omega_p(M_k)\leq p\,\omega_1(M_k)=2\pi p.\]
Moreover, by Lemma \ref{lem:M_k_prop}, $\lim_{k\to\infty}\omega_p(M_k)=2\pi p$.

Now let $p\geq 2$ be fixed and assume there is a sequence $\{\Gamma_k\}_{k\geq 1}\subset\mathcal{G}(M_k)$ such that:
\begin{enumerate}[label=(\roman*)]
    \item $\Gamma_k$ achieves $\omega_p(M_k)$, so 
    \[\length(\Gamma_k)=\omega_p(M_k);\]
    \item the sequence of stationary 1-varifolds $\{|\Gamma_k|\}_{k\geq 1}$ is not eventually constant and equal to $p|\gamma_0|$. 
\end{enumerate}
Without loss of generality, we can assume that $\Gamma_k\in\mathcal{G}_{\mathrm{prim}}(M_k)$ for all $k$. 
We shall now show that (ii) leads to a contradiction.
Since $\length(\Gamma_k)\leq 2\pi p$ for all $k$, 
we can apply Lemma \ref{lem:compact_var_met} and extract a smooth subsequential limit $\Gamma=\{\gamma^1, \dots, \gamma^N\}\in\mathcal{G}(S^1(1)\times\R)$. Therefore, along a converging subsequence (which we shall not relabel), for all sufficiently large $k$, we have $\Gamma_k=\{\gamma_k^1,\dots, \gamma_k^N\}$, and $\gamma_k^j\to\gamma^j$ smoothly for all $j=1,\dots, N$.
However, it is easy to see that if $\sigma:S^1\to S^1(1)\times\R$ is a geodesic loop in $S^1(1)\times\R$, then $\sigma(S^1)=S^1(1)\times\{t\}$ for some $t\in\R$. Therefore, for each $j=1,\dots, N$, $\gamma^j(S^1)=S^1(1)\times\{t_j\}$ for some $t_j\in\R$. 

By Lemma \ref{lem:M_k_prop}, $\gamma_k^j(S^1)$ intersects $\gamma_0(S^1)$ for all $j,k$, and, since $\gamma_k^j\to\gamma^j$, it must be the case that $t_j=0$, i.e. $\gamma^j(S^1)=\gamma_0(S^1)$ for all $j=1, \dots, N$. 

In particular, we have shown that for all $j=1,\dots, N$, $\gamma^j_k(S^1)$ converges to $\gamma_0(S^1)$ with some multiplicity.

Since, by (ii), the sequence of stationary 1-varifolds $\{|\Gamma_k|\}_{k\geq 1}$ is not eventually constant, there exist $j\in\{1,\dots, N\}$ and a further subsequence (again, not relabelled) along which $\gamma_k^j(S^1)\neq \gamma_0(S^1)$. Note that since $\gamma_k^j(S^1)\cap \gamma_0(S^1)\neq \emptyset$, $\gamma_k^j(S^1)$ must intersect $\gamma_0(S^1)$ transversely for all $k$. By Lemma \ref{lem:Jacobi_fields}, $\gamma_0$ admits a nontrivial Jacobi field with respect to the limit metric on $S^1(1)\times\R$. However, by a direct application of the second variation formula in \ref{prop:variation_formulas}, one can easily check all such Jacobi fields are easily seen to be constant. Hence, for all sufficiently large $k$, $\gamma_k^j(S^1)$ must lie to one side of $\gamma_0(S^1)$, which contradicts the previous statement that they must intersect transversely for all sufficiently large $k$.

We have thus shown that, if $\Gamma_k$ achieves $\omega_p(M_k)$, then the sequence of stationary 1-varifolds $\{|\Gamma_k|\}_{k\geq 1}$ must eventually be constant and equal to $p|\gamma_0|$, concluding the proof of Theorem \ref{thm:higher_mult_example}.
\end{proof}

\begin{remark}
By choosing $\mu=1$ or $\mu>1$, Theorem \ref{thm:higher_mult_example} provides examples where multiplicity must occur for both stable and unstable geodesics realising a min-max width for the length functional. When $\mu = 1$, we also note that $K > 0$, yet $\omega_p = p \omega_1$ for any $p$ finite (assuming $k$ sufficiently large). This contrasts with the work of Haslhofer--Ketover \cite{haslhofer2019minimal}, where it was shown that, in the setting of minimal $S^2$ in $S^3$, $\Ric > 0$ implies $\omega_2 < 2 \omega_1$. 
\end{remark}

\begin{remark}
Again, by \S \ref{GuthEx}, we know that $\omega_p \leq p \omega_1$, however Theorem \ref{thm:higher_mult_example} provides manifolds for which the above bound is sharp, in the sense that it can be saturated for all $p \leq p_0$ for any $p_0$ finite.
\end{remark}

\subsection{Extension to an open neighbourhood} 
We now verify \Cref{HigherMultCounter}. Let us now fix $\mu=1$. For each $p\in\N$, let $g_p$ be the smooth Riemannian metric on $S^2$ induced by the embedding \eqref{eq:M_k} of $M_{k_p}$ into $\R^3$, where $k_p$ is as in Theorem \ref{thm:higher_mult_example}, and let $\gamma_0=\gamma_0^{g_p}$ be the unique simple closed geodesic in $(S^2, g_p)$ corresponding to \eqref{eq:gamma_0}. 

\begin{remark}\label{rmk:nondegenerate}
Note that by direct computation of the Jacobi operator on $\gamma_0$, and by replacing $k_p$ with $k_p+1$ if needed, we can assume that $\gamma_0$ and its iterates are non-degenerate geodesics with respect to $g_p$, i.e. they admits no nontrivial Jacobi fields.

Indeed, by Proposition \ref{prop:variation_formulas}, the second variation of length of a parametrized geodesic loop $\gamma$ in the direction of a vector field $X$ along $\gamma$ only depends on the normal projection of $X$ along $\gamma$. Therefore, by restricting to normal vector fields $X=\phi\nu$, where $\nu$ is a choice of unit normal field along $\gamma$, the Jacobi operator $J_{\gamma, g_p}$ is given by the scalar operator
\[J_{\gamma, g_p}(\phi)=\phi''+K_{g_p}(x)\phi.\]
If $\gamma(S^1)=\gamma_0(S^1)$, then $K_{g_p}=k_p^{-1}$ along the image of $\gamma$. Therefore, the equation 
\[\phi''+k_p^{-1}\phi=0\]
has no nontrivial periodic solutions $\phi:[0,\length(\gamma)]\to\R$ provided $k_p^{-1/2}\,\length(\gamma)\notin\pi\Z$, which clearly holds if $k_p^{1/2}\notin\Z$.
\end{remark}

The main result of this subsection is the following theorem.

\begin{restatable}{thmm}{thm:higher_mult_nbhd} \label{thm:higher_mult_nbhd}
For each $\in\N$, there is an open neighbourhood $U_p$ of $g_p$ in the space of smooth Riemannian metrics on $S^2$ such that for all $g\in U_p$ the first $p$ widths of the length functional satisfy
\[\omega_l(S^2, g)=l\,\omega_1(S^2, g)\text { for } l=1, \dots, p,\]
and $\omega_l(S^2, g)$ can only be achieved by $l\,|\gamma^g_0|$, i.e. the stationary 1-varifold induced by the image $\gamma^g_0(S^1)$ of a simple closed geodesic $\gamma^g_0$ in $(S^2,g)$ with multiplicity $l$.
\end{restatable}
\begin{proof}
Fix $p\in\N$ and let $\{g^i\}_{i\in\N}$ be a sequence of smooth Riemannian metrics on $S^2$ converging smoothly to $g_p$. By Property (1) in Lemma \ref{lem:M_k_prop}, for all large enough $i$, the Gauss curvature $K_{g^i}$ of $(S^2,g^i)$ is strictly positive everywhere. Therefore, by \cite{CalabiCao92}, the first width $\omega_1(S^2, g^i)$ of length is achieved by a simple closed geodesic $\gamma_0^{g_i}$. Note that $\lim_{i\to\infty}\omega_1(S^2, g^i)=\omega_1(S^2,g_p)=\length(\gamma_0)$.
By Theorem \ref{thm:compactness}, up to a subsequence, $\{\gamma_0^{g^i}\}_i$ converges to a geodesic on $(S^2,g_p)$ achieving $\omega_1(S^2,g_p)$. Hence, by Theorem \ref{thm:higher_mult_example}, up to a subsequence $\gamma_0^{g^i}\to\gamma_0$ smoothly. 
Now consider $l\in\{2,\dots, p\}$. Note that $\omega_l(S^2,g^i)\to\omega_l(S^2,g_p)$. Let $\Gamma^i\in\mathcal{G}(S^2,g^i)$ achieve $\omega_l(S^2,g^i)$, so that
\[\omega_l(S^2,g^i)=\length(\Gamma^i).\] 
By Theorem \ref{thm:compactness}, up to a subsequence, 
\begin{equation}\label{eq:conv1}
\Gamma^i\to\Gamma\in\mathcal{G}(S^2, g_p),
\end{equation}
where $\Gamma$ achieves $\omega_l(S^2,g_p)$. By Theorem \ref{thm:higher_mult_example}, as a 1-varifold 
\begin{equation}\label{eq:conv2}
|\Gamma|=l\,|\gamma_0|.
\end{equation}
If $|\Gamma^i|$ is not of the form $l\,|\gamma^{g^i}|$ for some closed geodesic $\gamma^{g_i}$ for all sufficiently large $i$, then we can find a sequence $\{(\gamma_1^i,\gamma^i_2)\}$ of pairs of elements of $\Gamma^i$ such that $\gamma_1^i(S^1)\neq\gamma_2^i(S^1)$. Because of \eqref{eq:conv1} and \eqref{eq:conv2}, Lemma \ref{lem:Jacobi_fields} implies that there is a nontrivial $g_p$-Jacobi field on some iterate of $\gamma_0$, but this is a contradiction (see Remark \ref{rmk:nondegenerate}). Therefore, $|\Gamma^i|$ must be of the form $p\,|\gamma^{g^i}|$ for some closed geodesic $\gamma^{g_i}$ converging to $\gamma_0$, for all sufficiently large $i$. Moreover, if $\gamma^{g^i}(S^1)\neq\gamma_0^{g^i}(S^1)$, then the same argument again produces a nontrivial $g_p$-Jacobi field on an iterate of $\gamma_0$. Hence, we must have
\[|\Gamma^i|=l\,|\gamma_0^{g^i}|\]
for all sufficiently large $i$.
Since this property holds for all sequences $\{g^i\}_{i\in\N}$ of smooth Riemannian metrics on $S^2$ converging smoothly to $g_p$, this concludes the proof of the Theorem.
\end{proof}

We now prove \Cref{rptwoCorr}.
\rptwoCorr*
\begin{proof}
The proof is analogous to the proofs of \Cref{thm:higher_mult_example} and Theorem \ref{thm:higher_mult_nbhd} after quotienting by the antipodal map. We sketch the details
\begin{itemize}
    \item Let $N_k = M_k / \{(x,y,z) \sim (-x,-y,-z)\}$ with the induced metric, $\overline{g}_k$. Let $\overline{\gamma}_0$ correspond to the one-sided geodesic at $z = 0$.
    \item As $k \to \infty$, we see that 
    \begin{equation} \label{quotientWidths}
    \lim_{k \to \infty} \omega_p(N_k) = \omega_p((S^1 \times \R) / \{z \sim -z\}) = 2 \pi p
    \end{equation}
    This follows by noting that $(S^1 \times \R) / \{p \sim -p\} \supseteq S^1 \times (0, \infty)$ and applying the same argument as in \cite[\S 2]{WangZhou_mult2} or \cite[Lemma 6]{song2018existence}.
    %
    \item Let $\{\gamma_{i,k,p}\}$ be the geodesics such that 
    \[
    \omega_p(N_k) = \sum_{i = 1}^{N_{p,k}} m_{i,k,p} \cdot \mathrm{length}(\gamma_{i,k,p})
    \]
    Since the slices $\{z = t\} \sim \{z = -t\}$ are still mean convex, we argue that $\{\gamma_{i,k,p}\}$ must intersect the slice $\{z = 0\}  = \overline{\gamma}_0$. We similarly show that $\gamma_{i,k,p} \neq \overline{\gamma}_0$ for some $i$ would lead to the presence of a Jacobi field for $k$ sufficiently large. 
    \item Because $\mathrm{length}(\overline{\gamma}_0) = \pi$ and the widths converge for fixed $p$ via equation \eqref{quotientWidths}, we get multiplicity $2p$ exactly for $p$ fixed and all $k$ sufficiently large. This proves that 
    \[
    \omega_{l}(N_k) = \pi (2l)
    \]
    for all $l \leq p$ for all $k$ sufficiently large.
    \item To extend to an open neighborhood, repeat the argument in \Cref{thm:higher_mult_nbhd}. Since $\overline{\gamma}_0$ is non-degenerate, the moduli space of geodesics in an open neighborhood of $(N_k, \overline{g}_k)$ contains only one primitive geodesic. Via the continuity in the $p$-widths, we conclude the existence of an open neighborhoood of metrics, $\overline{U}$, where $\omega_p(\R \mathbb{P}^2, g) = p \omega_1(\R \mathbb{P}^2, g)$ for all $g \in \overline{U}$.
\end{itemize}
\end{proof}
\subsection{Ellipsoids close to the round sphere}
%
In this section, we provide another example of multiplicity on an open neighborhood of metrics based at the ellipsoid
\[
E(a_1, a_2, a_3) = \{(x_1, x_2, x_3) \in \R^3 \; : \; a_1 x_1^2 + a_2 x_2^2 + a_3 x_3^2 = 1\} \subseteq \R^3
\]
using the three geodesics
\[
\gamma_i(a_1, a_2, a_3) := E(a_1, a_2, a_3) \cap \{x_i = 0\}, \;\; i = 1,2,3
\]
The following theorem is implicit in \cite{chodosh2023p}, and heavily reliant on their Theorem \ref{CMGeodesicThm}, though we record it here for full clarity:
\begin{restatable}{thmm}{ellipEx} \label{ellipEx}
For any $P > 0$ and $k > 0$, there exists an $a_1 < a_2 < a_3$ sufficiently close to $1$, along with a $C^k$ open neighborhood $U \ni E(a_1,a_2,a_3)$ such that 
\[
\forall p < P, \forall g \in U \;\; \omega_p(g) = \sum_{i = 1}^3 m_i \mathrm{length}_g(\sigma_{i,g})
\]
where $\{\sigma_{i,g}\}_{i = 1}^3$ are geodesics with respect to  $g$.
\end{restatable}
\begin{proof}
From Morse \cite[Theorems IX 3.3, 4.1]{morse1934calculus}, it was shown that for every $\Lambda > 2 \pi$, there exists an $\eps > 0$ so that if $a_1, a_2, a_3 \in [1-\eps, 1 + \eps]$ and $a_1 < a_2 < a_3$, then every closed connected immersed geodesic $\gamma \subseteq E(a_1, a_2, a_3)$ with $\mathrm{length}(\gamma) < 2 \Lambda$ is non-degenerate. Moreover, any such $\gamma$ is a multiple of $\gamma_i(a_1, a_2, a_3)$. By taking $\eps$ sufficiently small, we have that 
\[
\forall p < P, \qquad |\omega_p(E(a_1,a_2,a_3)) - \omega_p(S^2, g_{round})| \leq o(1)
\]
In particular, from \cite{chodosh2023p}
\[
\omega_p(S^2, g_{round}) = 2 \pi \lfloor \sqrt{p} \rfloor
\]
Choose $\Lambda$ so that $\Lambda > 2 \pi \lfloor \sqrt{p} \rfloor$, which provides an upper bound for the length of a single immersed geodesic realizing the $p$-width. We now claim that there exists an open neighborhood $U \ni E(a_1, a_2, a_3)$ such that 
\begin{enumerate}
    \item Any $\gamma$ with $\mathrm{length}(\gamma) < 1.5 \Lambda$ is non-degenerate
    \item Any $\gamma$ with $\mathrm{length}(\gamma) < 1.5 \Lambda$ is a multiple of one of $3$ geodesics $\sigma_i$
\end{enumerate}
We know that for $U$ sufficiently small, for each $g \in U$, there exists $3$ such geodesics $\overline{\sigma}_i$, such that 
\begin{equation} \label{lengthOfPerturbed}
|\mathrm{length}_g(\overline{\sigma}_i) - \mathrm{length}_{g_{a_1,a_2,a_3}}(\gamma_i)| \leq K ||g - g_{a_1,a_2, a_3}||_{C^2}
\end{equation}
and $\overline{\sigma}_i$ is non-degenerate. This follows from the inverse function theorem and non-degeneracy being an open condition (see also \cite{white1991space} Theorem 2.1). Suppose no such open subneighborhood exists for which property $(2)$ holds. Then there exists a sequence of metric $g_{l}$ such that $||g_{a_1, a_2, a_3} - g_{l}||_{C^k} \to 0$, along with $\sigma_{l}$ such that $\mathrm{length}(\sigma_{l}) < 1.5$ and $\sigma_{l} \not \in \{\overline{\sigma}_1,\overline{\sigma}_2, \overline{\sigma}_3\}$. By Lemma \ref{lem:compact_var_met}, we know that $\sigma_{l} \xrightarrow{C^2} \rho = m_0 \cdot \gamma_i$ for some $l$ and some $i \in \{1,2,3\}$. Without loss of generality, $i = 1$. Now for each $l$, consider the pair $(\sigma_{l}, m_0 \overline{\sigma}_1)$. By Lemma \ref{lem:Jacobi_fields}, there exists a Jacobi field supported on $m_0 \cdot \gamma_i$, a contradiction.

%
%
Applying Theorem \ref{CMGeodesicThm}, we know that 
\[
\forall g \in U,\ \forall p < P, \quad \omega_p(g) = \sum_{i = 1}^3 m_i(g) \mathrm{length}_g(\tilde{\sigma}_i)
\]
\end{proof}
Recalling the closeness of $\mathrm{length}_g(\tilde{\sigma}_i)$ to $\mathrm{length}(\gamma_i)$ in \eqref{lengthOfPerturbed}, along with the closeness of widths to those of the round metric (see \cite{marques2019equidistribution}, Lemma 1.1) we see that $m_i$ cannot all be $1$ if $p > 3$.

\section{Appendix}
We record some miscellaneous restrictions on the $p$-widths. We recall the historic result of Calabi--Cao
\begin{restatable}[Calabi--Cao, Thm D \cite{CalabiCao92}]{thmm}{CCThm} \label{CCThm}
If $g$ is a $C^3$ smooth metric on a two-sphere with non-negative curvature, then $\omega_1$ is achieved by a simple curve.
\end{restatable}
Here, \textit{simple} means no self-intersections. As a byproduct of this theorem and also Guth's construction of a $p$-sweepout \ref{GuthEx}, we conclude the following
\begin{restatable}{corr}{weakMultBound} \label{weakMultBound}
Let $\omega_1 = \mathrm{length}(\gamma_1)$. Suppose $\omega_p$ is achieved by
\[
\gamma_p = \sum_{i = 1}^{N_p} m_{i,p} \gamma_{i,p}
\]
Then $\sum_{i = 1}^{N_p} m_{i,p} \leq p$, and equality holds if and only if $\mathrm{length}(\gamma_{i,p}) = \mathrm{length}(\gamma_1)$
\end{restatable}
\begin{proof} 
Apply Guth's $p$-sweepout construction to the optimal $1$-sweepout, $p$ times as a competitor. \end{proof}
We can also leverage \textit{bounded} curvature to give injectivity radius bounds. This ends up allowing us to restrict the number of edges. Again, let $\gamma_p = \sum_{i = 1}^{N_p} m_{i,p} \gamma_{i,p}$ be the union of geodesics achieving $\omega_p$. Consider the graph, $G(\gamma_p)$, where a vertex corresponds to $x \in \Gamma$ such that $\ord(x) \geq 2$ and the edges correspond to the geodesics (contained in $\gamma_p$) between two points $x,y$ with $\ord(x), \ord(y) \geq 2$.  In this definition, we will count edges \textit{with multiplicity}, but vertices \textit{without multiplicity}.
\begin{restatable}{prop}{edgeBound} \label{edgeBound}
Suppose $(M^2, g)$ with $1 \geq K \geq c_0$. Let $e_G$ be the number of edges in $G(\gamma_p)$ as above. Then 
\[
e_G \leq \frac{p \cdot \omega_1}{\pi}
\]
\end{restatable}
\begin{proof}
For manifolds with $K \leq 1$, the injectivity radius is at least $\pi$ by the Rauch comparison theorem. Thus 
\[
\mathrm{length}(\gamma_p) = \sum_{e \in G(\gamma_p)} \mathrm{length}(e) \geq e_G \cdot \pi 
\]
using the upper bound of $\omega_p \leq p \cdot \omega_1$. \end{proof} 
We note that a similar bound can be deduced from the index bound \Cref{thm:index_upper}, and lower bounds on curvature.
\begin{restatable}{prop}{lengthUpper} \label{lengthUpper}
Suppose $(M^2, g)$ with $K \geq K_0 > 0$. Then any $\{\gamma_{i,p}\}$ can be length at most $\frac{\pi p}{\sqrt{K_0}}$.
\end{restatable}
\begin{proof}
Consider the second variation among normal perturbations of $\gamma_{i,p}$
\[
Q(f, f) = \int_{\gamma_{i,p}} |\nabla f|^2 - K |f|^2
\]
Consider $f_k = \sin ( k \pi x / \ell_{i,p})$ where $\ell_{i,p} = \mathrm{length}(\gamma_{i,p})$, then
\[
Q(f_k, f_k) \leq C \cdot [ k^2 \pi^2 / \ell_{i,p}^2 - \inf_{\gamma_{i,p}} K ]
\]
When $k = p$, the above must be $\geq 0$, else $\{1, f_1, \dots, f_{p}\}$ would contradict the index bound of \Cref{thm:index_upper}. Thus, we conclude
\[
\ell_{i,p} \leq \frac{\pi p}{K_0}
\]
\end{proof}
These propositions may be useful to rule out the presence of geodesic flowers (see e.g. \cite{chambers2023geodesic}), though one may hope do this by connecting $\omega_1$ to $K_0$ and then applying Guth's $p$-sweepout construction.
\bibliographystyle{amsalpha}
\bibliography{bib.bib}

\end{document}